\documentclass[11pt,leqno]{amsart}
\usepackage{amscd}
\usepackage{amssymb}
\usepackage{amsfonts}
\usepackage[centertags]{amsmath}
\usepackage{latexsym}
\usepackage{verbatim}
\usepackage{amsthm}
\usepackage{color}
\usepackage{pgfplots}
\usepackage{tikz}
\usetikzlibrary{intersections}
\usepackage{amsfonts}
\usepackage{amssymb}
\usepackage{amsmath}
\usepackage{amsthm}
\usepackage[english]{babel}
\usepackage{enumerate}
\usepackage{graphicx}
\usepackage{color}

\newcommand{\NN}{\mathbb{N}}
\newcommand{\RR}{\mathbb{R}}

\newcommand{\Sbb}{\mathbb{S}}

\newcommand{\BRT}{B_r\cap T_pS}
\newcommand{\Om}{\Omega}
\newcommand{\om}{\omega}
\newcommand{\pa}{\partial}

\newcommand{\al}{\alpha}
\newcommand{\si}{\sigma}
\newcommand{\ga}{\gamma}
\newcommand{\zi}{\zeta}
\newcommand{\de}{\delta}

\newcommand{\ep}{\varepsilon}
\newcommand{\lam}{\lambda}

\newcommand{\ints}{\int\limits}

\newtheorem{theorem}{Theorem}[section]
\newtheorem{corollary}[theorem]{Corollary}
\newtheorem{lemma}[theorem]{Lemma}
\newtheorem{proposition}[theorem]{Proposition}

\newtheorem{remark}[theorem]{Remark}

\newtheorem{thmx}{Theorem}

\theoremstyle{remark}

\DeclareMathOperator{\diver}{\mathrm{div}}

\DeclareMathOperator{\dist}{\mathrm{dist}}
\DeclareMathOperator{\osc}{\mathrm{osc}}
\DeclareMathOperator{\diam}{\mathrm{diam}}

\newcommand{\oscH}{\osc(H)}

\renewenvironment{proof}{
  \noindent{\it Proof.}\ }{\hspace*{\fill}
  \begin{math}\Box\end{math}\medskip}

\title[A quantitative Alexandrov's theorem]{A sharp quantitative version of Alexandrov's theorem via the method of moving planes}

\author{Giulio Ciraolo, Luigi Vezzoni}

\date{\today}
\address{Dipartimento di Matematica e Informatica, Universit\`a di Palermo, Via Archirafi 34, 90123, Italy} \email{giulio.ciraolo@unipa.it}

\address{Dipartimento di Matematica G. Peano, Universit\`a di Torino, Via Carlo Alberto 10, 10123 Torino, Italy.} \email{luigi.vezzoni@unito.it}

\keywords{Alexandrov Soap Bubble Theorem, method of moving planes, stability, mean curvature, pinching.}
    \subjclass{Primary 35B50, 35B51; Secondary 53C20, 53C21}

\begin{document}

\maketitle

\begin{abstract}
We prove the following quantitative version of the celebrated \emph{Soap Bubble Theorem} of Alexandrov. 
Let $S$ be a $C^2$ closed embedded hypersurface of $\RR^{n+1}$, $n\geq1$, and denote by $\oscH$ the oscillation of its mean curvature. We prove that there exists a positive $\ep$, depending on $n$ and upper bounds on the area and the $C^2$-regularity of $S$, such that if $\oscH \leq \ep$ then there exist two concentric balls $B_{r_i}$ and $B_{r_e}$ such that $S \subset \overline{B}_{r_e} \setminus B_{r_i}$ and $r_e -r_i \leq C \oscH$, with $C$ depending only on $n$ and upper bounds on the surface area of $S$ and the $C^2$ regularity of $S$. Our approach is based on a quantitative study of the method of moving planes and the quantitative estimate on $r_e-r_i$ we obtain is optimal. 

As a consequence of this theorem, we also prove that if $\oscH$ is small then $S$ is diffeomorphic to a sphere and give a quantitative bound which implies that $S$ is $C^1$-close to a sphere. 


\end{abstract}

\tableofcontents

\section{Introduction} \label{section introd}
The {\it Soap Bubble Theorem} proved by Alexandrov in \cite{Al1} has been the object of many investigations. In its simplest form it states that

\begin{center}

\emph{The $n$-dimensional sphere is the only compact connected embedded hypersurface of $\RR^{n+1}$}

\emph{with constant mean curvature.}

\end{center}

\noindent As it is well-known, the embeddedness condition is necessary, as implied by the celebrated counterexamples by Hsiang-Teng-Yu \cite{HTY} and Wente \cite{We}. There have been several extensions of the rigidity result of Alexandrov to more general settings. Alexandrov proved this Theorem in a more general setting; in particular, the Euclidean space can be replaced by any space of constant curvature (see also \cite{Al2} where he discussed several possible generalizations).  Montiel and Ros \cite{MR} and Korevaar \cite{Ko} studied the case of hypersurfaces with constant higher order mean curvatures embedded in space forms. Alexandrov Theorem has been studied also for warped product manifolds by Montiel \cite{Mo},  Brendle \cite{Br} and Brendle and Eichmair \cite{BE}. There are many other related results; the interested reader can refer to \cite{CFSW,CFMN,CY,DCL,HY,Re,Ro1,Ro2,Ya} and references therein.

To prove the Soap Bubble Theorem, Alexandrov introduced the \emph{method of moving planes}, a very powerful technique which has been the source of many insights in analysis and differential geometry. Serrin understood that the method can be applied to Partial Differential Equations. Indeed, in his seminal paper \cite{Se} he obtained a symmetry result for the torsion problem which gave rise to a huge amount of results for overdetermined problems (the interest reader can refer to the references in \cite{CMS}). In \cite{GNN} Gidas, Ni and Nirenberg refined Serrin's argument to obtain several symmetry results for positive solutions of second order elliptic equations in bounded and unbounded domains (see also \cite{Li1} and \cite{Li2}). The method was further employed by Caffarelli, Gidas and Spruck \cite{CGS} to prove asymptotic radial symmetry of positive solutions for the conformal scalar curvature equation and others semilinear elliptic equations (see also \cite{KMPS}). The moving planes were also used to obtain several celebrated results in differential geometry: Schoen \cite{Sc} characterized the catenoid, Meeks \cite{Me} and Korevaar, Kusner and Solomon \cite{KKS} showed that a complete connected properly embedded constant mean curvature surface in the Euclidean space with two annuli ends is rotationally symmetric. There is a large amount of other interesting papers on these topics which are not mentioned here.

Alexandrov's proof in the Euclidean space works as follows: (i) show that for any direction $\om$ there exists a critical hyperplane orthogonal to $\om$ which is of symmetry for the surface $S$; (ii) since the center of mass $\mathcal{O}$ of $S$ lies on each hyperplane of symmetry, then every hyperplane passing through $\mathcal{O}$ is of reflection symmetry for $S$; (iii) since any rotation about $\mathcal{O}$ can be written as a composition of $n+1$ reflections, then $S$ is rotationally invariant, which implies that $S$ is the $n$-dimensional sphere.
The crucial step in this proof is (i), which is obtained by applying the method of moving planes and using maximum principle (see Theorem A in Subsection \ref{subsection Alexandrov}).

In this paper we study a quantitative version of the Soap Bubble Theorem, that is we assume that the oscillation of the mean curvature $\osc(H)$ is \emph{small} and we prove that $S$ is \emph{close to a sphere}. 
More precisely, let $S$ be an $n$-dimensional, $C^2$-regular, connected, closed hypersurface embedded in $\RR^{n+1}$, and denote by $|S|$ the area of $S$. Since $S$ is $C^2$ regular, then it satisfies a uniform touching sphere condition of (optimal) radius $\rho$.
We orientate $S$ according to the inner normal. Given $p \in S$, we denote by $H(p)$ the mean curvature of $S$ at $p$, and we let 
$$\oscH = \max_{p\in S} H(p) - \min_{p\in S} H(p).$$

\medskip
Our main result is the following theorem.

\begin{theorem}\label{main}
Let $S$ be an $n$-dimensional, $C^2$-regular, connected, closed hypersurface embedded in $\RR^{n+1}$. There exist constants $\ep,\, C>0$ such that if
\begin{equation}\label{H quasi const}
\oscH \leq \ep,
\end{equation}
then there are two concentric balls $B_{r_i}$ and $B_{r_e}$ such that
\begin{equation}\label{Bri Om Bre}
S \subset \overline{B}_{r_e} \setminus B_{r_i},
\end{equation}
and
\begin{equation}\label{stability radii}
r_e-r_i \leq C \oscH.
\end{equation}
The constants $\ep$ and $C$ depend only on $n$ and upper bounds on $\rho^{-1}$ and $|S|$. 
\end{theorem}

Under the assumption that $S$ bounds a convex domain, there exist some results in the spirit of Theorem \ref{main} which are available in literature. In particular, when the domain is an ovaloid, the problem was studied by Koutroufiotis \cite{Kou},  Lang \cite{La} and Moore \cite{Moo}. Other stability results can be found in Schneider \cite{Sc} and Arnold \cite{Ar}. These results were improved by Kohlmann in \cite{Kol} where he proved an explicit H\"{o}lder type stability in \eqref{stability radii}. In Theorem \ref{main}, we do not consider any convexity assumption and we obtain the optimal rate of stability in \eqref{stability radii}, as can be proven by a simple calculation for ellipsoids. 

Theorem \ref{main} has a quite interesting consequence which we explain in the following. It is well-known (see for instance \cite{Gr}) that if every principal curvature $\kappa_i$ of $S$ is pinched between two positive numbers, i.e.
$$ 
\frac{1}{r} \leq \kappa_i \leq (1+\delta) \frac{1}{r}, \quad i=1,\ldots,n,
$$
then $S$ is close to a sphere of radius $r$. Following Gromov \cite[Remark (c), p.67--68]{Gr}, one can ask what happens when only the mean curvature is pinched. We have the following result.

\begin{corollary}\label{main2}
Let $\rho_0,A_0>0$ and $n \in \NN$ be fixed. There exists a positive constant $\ep$, depending on $n$, $\rho_0$ and $A_0$, such that if $S$ is a connected closed $C^2$ hypersurface embedded in $\RR^{n+1}$ with $|S|\leq A_0$, $\rho \geq \rho_0$, whose mean curvature $H$ satisfies 
\begin{equation*}
\oscH <\ep\,,
\end{equation*}
then $S$ is diffeomorphic to  a sphere.   
 
Moreover $S$ is $C^1$-close to a sphere, i.e. there exists a $C^1$-map $F=Id + \Psi \nu: \partial B_{r_i} \to S$ such that 
\begin{equation} \label{Lipschitz_bound}
\|\Psi\|_{C^1(\partial B_{r_i})} \leq C (\oscH)^{\frac{1}{2}} \,,
\end{equation}
where $C$ depends only on $n$ and upper bounds on $\rho^{-1}$ and $|S|$.
\end{corollary}

%
%
%
%

Before explaining the argument of the proof of Theorem \ref{main}, we give a couple of remarks on the bounds on $\rho$ and $|S|$ in Theorem \ref{main} and its Corollary \ref{main2}. The upper bound on $\rho^{-1}$ controls the $C^2$ regularity of the hypersurface, which is a crucial condition for obtaining an estimate like \eqref{stability radii}. 
Indeed, if we assume that $\rho$ is not bounded from below, it is possible to construct a family of closed surfaces embedded in $\RR^3$, not diffeomorphic to a sphere, with $\osc(H)$ arbitrarly small and such that \eqref{stability radii} fails (see Remark \ref{remark ros} and \cite{CirMag}).
The upper bound on $|S|$ is a control on the constants $\ep$ and $C$, which clearly change under dilatations.

We remark that Corollary \ref{main2} can be obtained by a compactness argument by using the theory of varifolds by Allard \cite{All} and Almgren \cite{Alm}.  Indeed, by Allard's compactness theorem every sequence of closed hypersurfaces satisfying (uniformly) the assumptions of Corollary \ref{main2} admits a
subsequence which, up to translations, converges to a hypersurface which satisfies a touching ball condition and hence is $C^{1,1}$ regular. By standard regularity theory, the hypersurface
is smooth and is a sphere by the classical Alexandrov theorem. We think that also the stability estimates in Theorem \ref{main} can be obtained by using Allard's regularity theorem.

There are other possible strategies to obtain quantitative estimates for almost constant mean curvature hypersurfaces and give results as in the spirit of Theorem \ref{main}. Indeed, as we already mentioned, there are several proofs of the rigidity result of Alexandrov (i.e. when $H$ is constant). Beside the method of moving planes (which will be our approach), one could try to quantitavely study the proofs in \cite{MR}, \cite{Re} and \cite{Ro2}, which are based on integral identities. For instance, the approach in \cite{CirMag} starts form \cite{Ro2} and finds quantitative estimates on the closedness of the hypersurface to a compound of tangent balls. As explained in \cite[Appendix A]{CirMag}, another possible approach would be to start from the proof in \cite{MR} and then study almost umbilical hypersurfaces, in the same spirit as \cite{DeLMul1} and \cite{DeLMul2}. However, these approaches based on integral identities do not seem to lead to optimal estimates as in our Theorem \ref{main} (see \cite{CirMag} for a detailed discussion).

Our approach, instead, is based on a quantitative analysis of the method of moving planes and uses arguments from elliptic PDEs theory. Since the proof of symmetry is based on maximum principle, our proof of the stability result will make use of Harnack's and Carleson's (or boundary Harnack's) inequalities and the Hopf Lemma, which can be considered as the quantitative counterpart of the strong and boundary maximum principles. We emphasize that the stability estimate \eqref{stability radii} is optimal and that our proof permits to compute the constants explicitly.


A quantitative study of the method of moving planes was first performed in \cite{ABR}, where the authors obtained a stability result for Serrin's overdetermined problem \cite{Se}, and it has been used in a series of paper by the first author \cite{CMS2,CMV,CMV2} for studying the stability of radial symmetry for Serrin's and other overdetermined problems (see also \cite{BNST} for an approach based on integral identities). In this paper, we follow the same approach of \cite{ABR}, but the setting here is complicated by the fact that we have to deal with manifolds. As we will show, the main goal is to prove an approximate symmetry result for one (arbitrary) direction. After that, the approximate radial symmetry is well-established and follows by an argument in \cite{ABR}.
To prove the approximate symmetry in one direction, we apply the method of moving planes and show that the union of the maximal cap and of its reflection provides a set that \emph{fits well} $S$. This is the main point of our paper and is achieved by developing the following argument. Assume that the surface and the reflected cap are tangent at some point $p_0$ which is an interior point of the reflected cap, and write the two surfaces as graphs of function in a neighborhood of $p_0$. The difference $w$ of these two functions satisfies an elliptic equation $Lw=f$, where $\|f\|_\infty$ is bounded by $\oscH$. By applying Harnack's inequality and interior regularity estimates, we have a bound on the $C^1$ norm of $w$, which says that the two graphs are close in $C^1$ norm no more than some constant times $\oscH$. It is important to observe that this estimate implies that the two surfaces are close to each other and also that the two corresponding Gauss maps are close (in some sense) in that neighborhood of $p_0$.  Then we connect any point $p$ of the reflected cap to $p_0$ and we show that such closeness propagates at $p$. Since we are dealing with a manifold, we have to change local parameterization while we are moving from $p_0$ to $p$ and we have to prove that the closeness information is preserved. By using careful estimates and making use of interior and boundary Harnack's inequalities, we show that this is possible if we assume that $\oscH$ is smaller than some fixed constant. 

The paper is organized as follows. In Section \ref{section preliminary} we prove some preliminary results about hypersurfaces in $\RR^{n+1}$, we recall some results on classical solutions to mean curvature type equations, and we give a sketch of the proof of the symmetry result of Alexandrov. In Section \ref{section technical lemmas} we prove some technical lemmas which will be used for proving the stability result. In Sections \ref{section proof main} and \ref{section proof main2} we prove Theorem \ref{main} and Corollary \ref{main2}, respectively. 

\bigskip
\noindent {\em Acknowledgements.}
The authors wish to thank Manuel Del Pino, Francesco Maggi, Antonio Ros and Susanna Terracini for their enlightening remarks and useful discussions we had together. The problem faced in the paper was suggested to the first author by Rolando Magnanini, who is also acknowledged for his interest in the work. Morerover, the authors would like to thank also Bang-Yen Chen, Barbara Nelli, Paolo Piccione, Fabio Podest\`a and Paolo Salani for useful conversations. 
We also thank the anonymous referees for their remarks and for suggesting an improvement of Corollary \ref{main2}.

The paper was completed while the first author was visiting \lq\lq The Institute for Computational Engineering and Sciences\rq\rq (ICES) of The University of Texas at Austin, and he wishes to thank the institute for hospitality and support. The first author has been supported by the NSF-DMS Grant 1361122, the \lq\lq Gruppo Nazionale per l'Analisi Matematica, la Probabilit\`a e le loro
Applicazioni\rq\rq (GNAMPA) of the Istituto Nazionale di Alta Matematica (INdAM) and the project FIRB 2013 ``Geometrical and Qualitative aspects of PDE''.

The second author was supported by the project FIRB ``Geometria differenziale e teoria geometrica delle funzioni'',
 the project PRIN
\lq\lq  Variet\`a reali e complesse: geometria, topologia e analisi armonica" and by GNSAGA of INdAM.

\section{Notation and preliminary results} \label{section preliminary}

In this section we collect some preliminary results which will be useful in the following. Although some of them are already known, we sketch their proofs for sake of completeness and in order to explain the notation which it will be adopted in the sequel.

\medskip
Let $S$ be a $C^2$ regular, connected, closed hypersurface embedded in $\RR^{n+1}$, $n\geq 1$,  and let  $\Omega$ be the relatively compact domain of $\mathbb R^{n+1}$ bounded by $S$. We denote by $T_{p}S$ the tangent hyperplane to $S$ at $p$ and by $\nu_p$ the inward normal vector.
Given a point $\xi \in \RR^{n+1}$ and an $r>0$,  we denote by $B_r(\xi)$ the ball  in $\RR^{n+1}$ of radius $r$ centered at $\xi$. When a ball is centered at the origin $O$, we simply write $B_r$ instead of $B_r(O)$.
%

\medskip
Let ${\rm dist}_S\colon \RR^{n+1} \to \RR$ be the distance function from $S$, i.e.
\begin{equation*}
{\rm dist}_S(\xi)=
\begin{cases}
\dist(\xi,S), & \textmd{if } \xi \in \Om,\\
-\dist(\xi,S), & \textmd{if } \xi \in \RR^{n+1} \setminus \Om;
\end{cases}
\end{equation*}
it is clear that $S = \{\xi \in \RR^{n+1}:\ {\rm dist}_S(\xi)=0\}$. Moreover, it is well-known (see e.g. \cite{GT}) that $\dist_S$ is Lipschitz continuous with Lipschitz constant 1 and that it is of class $C^2$ in an open neighborhood of $S$.
Therefore the {\em implicit function theorem} implies that, given a point $p\in S$,  $S$ can be locally represented
as a graph over the tangent hyperplane $T_pS $: there exist an open neighbourhood $\mathcal{U}_r(p)$ of $p$ in $S$ and a $C^2$-map $u\colon B_r\cap T_pS \to \mathbb R$ such that 
\begin{equation}\label{urp}
\mathcal U_r(p) = \{p+x+u(x)\nu_p\,:\,\ x \in B_r\cap T_pS\}.
\end{equation}
Moreover, if $q=p+x  + u(x) \nu_p$, with $x \in B_r(p)\cap T_pS$, we have
\begin{equation}\label{nu nabla u}
\nu_q= \frac{\nu_p-\nabla u(x)}{\sqrt{1+ |\nabla u(x)|^2}},
\end{equation} 
where
$$
\nabla u(x)=\sum_{i=1}^N \partial_{e_i}u(x)\,e_i\,,
$$
and $\{e_1,\dots,e_n\}$ is an arbitrary {\em orthonormal } basis of $T_pS$. We notice that, according with the definition above, $\nabla u(x)$ is a vector in $\RR^{n+1}$ for every $x$ in the domain of $u$. Moreover
$\nu_q\cdot \nu_p>0$
for every $q\in \BRT$ and, if $|\nabla u|$ is uniformly bounded in $\BRT$, then $u$ can be extended to
$B_{r'}\cap T_pS$ with  $r'>r$.

Since $S$ is $C^2$-regular, then the domain $\Om$ satisfies a uniform touching ball condition and we denote by $\rho$ the optimal radius, that is: for any $p \in S$ there exist two balls of radius $\rho$ centered at $c^- \in \Om$ and $c^+ \in \RR^{n+1} \setminus \overline{\Om}$ such that $B_\rho(c^-) \subset \Om$, $B_\rho(c^+) \subset \RR^{n+1} \setminus \overline{\Om}$, and $p \in \partial B_\rho(c^\pm)$. $B_\rho(c^-)$ and $B_\rho(c^+)$ are called, respectively, the \emph{interior} and \emph{exterior} touching balls at $p$.
%
%


In the following Lemma we show that we may assume $r=\rho$ in the definition of \eqref{urp}, and we give some bounds in terms of $\rho$ which will be useful in the sequel.

\begin{lemma} \label{lemma bounds rho}
Let $p\in S$. There exists a $C^2$ map $u: B_{\rho}\cap T_pS \to \RR$
such that
\begin{equation*}
\mathcal U_\rho(p)= \{p+x+u(x)\nu_p\,:\,\ x \in B_{\rho}\cap T_pS\}
\end{equation*}
is a relative open set of $S$ and
\begin{eqnarray}
&& \label{bound on u} |u(x)| \leq \rho - \sqrt{\rho^2-|x|^2}\,,\\
&& \label{bound on nabla u} |\nabla u(x)| \leq \frac{|x|}{\sqrt{\rho^2-|x|^2}}\,,
\end{eqnarray}
for every $x \in B_{\rho}\cap T_pS$. Moreover 
\begin{equation} \label{bound on nu N+1}
\nu_p\cdot \nu_q\geq \frac{1}{\rho} \sqrt{\rho^2 - |x|^2}, \ \ \textmd{ and } \ \ |\nu_p - \nu_q| \leq \sqrt{2}\frac{|x|}{\rho},
\end{equation}
for every $q=p+x+u(x)\nu_p$ in $\mathcal U_\rho(p)$.
\end{lemma}

\begin{proof}
By the implicit function theorem, there exists $r>0$, $u:B_r\cap T_pS  \to \RR$ and $\mathcal{U}_r(p)$ as in \eqref{urp}. We may assume that $r \leq \rho$.
The bound \eqref{bound on u} in $B_r\cap T_pS$ easily follows from the definition of interior and exterior touching balls at $p$.
We prove that estimate \eqref{bound on nabla u} in $\BRT$, which allows us to enlarge the domain of
$u$ up to $B_{\rho}\cap T_pS $.
Let
$$
q=p+x+u(x)\nu_p\,,
$$
with $|x|<r$ be an arbitrary point of $\mathcal U_r(p)$  (notice that $\nu_p\cdot \nu_q>0$). Since
$$
B_\rho(p + \rho \nu_p) \cap B_\rho(q- \rho \nu_q) = \emptyset\,, 
$$
we have that
$$
|p + \rho \nu_p - q + \rho \nu_q |\geq 2\rho.
$$
Analogously, $B_\rho(p-\rho \nu_p) \cap B_\rho(q+ \rho \nu_q) = \emptyset$ gives that
$$
|q+\rho \nu_q - p + \rho \nu_p|\geq 2\rho.
$$
By adding the squares of the last two inequalities we obtain that
$$
|p-q|^2 + 2 \rho^2 (\nu_p\cdot \nu_q) \geq 2 \rho^2,
$$
and from \eqref{bound on u} we get \eqref{bound on nu N+1}.
From \eqref{nu nabla u} and \eqref{bound on nu N+1} we obtain \eqref{bound on nabla u} in $B_r\cap T_pS$. Since $|\nabla u|$ is bounded in $\overline{B}_r\cap T_pS$, then we can extend $u$ in a larger ball where \eqref{bound on nabla u} is still satisfied.  It is clear that we can choose $r=\rho$ and  \eqref{bound on u}--\eqref{bound on nu N+1} hold.
\end{proof}

Given $p,q \in S$ we denote by $d_S(p,q)$ their intrinsic distance inside $S$ and, if $A$ is an arbitrary subset of $S$, we define
\begin{equation*}
d_S (p,A)= \inf_{q \in A} d_S(p,q).
\end{equation*}
We have the following Lemma.

\begin{lemma} \label{lemma bound on d Ga}
Let $p \in S$, $q\in \mathcal{U}_\rho(p)$ and let $x$ be the orthogonal projection of $q$ onto the hyperplane $T_pS$. Then,
\begin{equation}
 \label{bound on d Ga above}  |x| \leq d_S(p,q) \leq \rho \arcsin \frac{|x|}{\rho}.
\end{equation}
\end{lemma}

\begin{proof}
The first inequality is trivial. In order to prove the second inequality we consider the curve  $\gamma\colon [0,1]\to S$ joining $p$ with $q$ defined by   $\ga(t)=p+tx+u(tx)\,\nu_p$, $t\in[0,1]$. Then
$$
\dot\gamma(t)=x+ ( \nabla u(tx)\cdot x)\,\nu_p;
$$
since $x \in T_p S$ and by Cauchy-Schwartz inequality we obtain that
$$
|\dot\gamma(t)|\leq |x| \sqrt{1+|\nabla u(tx)|^2}\,.
$$
Therefore inequality  \eqref{bound on nabla u}  in Lemma \ref{lemma bounds rho} implies
$$
 |\dot\gamma(t)|\leq \,\frac{\rho |x|}{\sqrt{\rho^2-t^2|x|^2}}.
$$
Since
\begin{equation*}
d_S(p,q) \leq \int_{0}^1 |\dot\gamma(t)| dt \, ,
\end{equation*}
then
\begin{equation*}
d_S(p,q) \leq |x| \rho \int_{0}^1  \frac{1}{\sqrt{\rho^2-t^2|x|^2}} dt 
\end{equation*}
which gives \eqref{bound on d Ga above}.
\end{proof}

Let $p\in S$ and let $u \colon B_\rho \cap T_p S \to S $ as in Lemma \ref{lemma bounds rho}. It is well-known (see \cite{GT}) that $u$ is a classical solution to
\begin{equation}\label{MCeq}
\diver \left(\frac{\nabla u}{\sqrt{1+|\nabla u|^2}} \right) = n H, \quad \textmd{in } B_\rho \cap T_pS,
\end{equation}
where $H$ is the mean curvature of $S$ regarded as a map on $ B_\rho \cap T_p S $. We notice that $\nabla u \in T_pS$ and the divergence is meant in local coordinates on $T_pS$: if $\{e_1,\ldots,e_n\}$ is an orthonormal basis of $T_pS$ and $F=\sum_{i=1}^n F_i e_i$, then
\begin{equation*}
\diver F = \sum_{i=1}^n \frac{\pa F_i}{\pa e_i}.
\end{equation*}
Moreover, \eqref{MCeq} is {\em uniformly elliptic} once $u$ is regarded as a regular map in an open set of $\RR^n$ and has bounded gradient, since
\begin{equation}\label{ellipticity MCeq}
|\xi|^2 \leq \frac{\pa}{\pa \zi_j} \left(\frac{\zi_i}{\sqrt{1+|\zi|^2}}\right) \xi_i \xi_j \leq (1+|\zi|^2) |\xi|^2.
\end{equation}
for every $\xi=(\xi_1\dots,\xi_n)$, $\zi=(\zi_1,\dots, \zi_n)$ in $\RR^n$.

\subsection{Classical solutions to mean curvature equation} In this subsection we collect some results about classical solutions to \eqref{MCeq} which will be used in the next sections. 

Let $B_r$ be the ball of $\RR^k$ centered at the origin and having radius $r$. Given a differentiable map $u:B_r \to \RR$, we denote by $Du$ the gradient of $u$ in $\RR^k$: 
$$
Du=\left( \frac{\pa u}{\pa x_1}, \ldots, \frac{\pa u}{\pa x_k} \right) \,.
$$   
We remark that this notation differs from the one in the rest of the paper, where we use the $\nabla$ symbol to denote a vector in $\RR^{n+1}$. 

Let $H_0,\,H_1 \in C^0(B_r)$ and $u_0$ and $u_1$ be two classical solutions of
\begin{equation}\label{MCeq uj}
\diver \left(\frac{D u_j}{\sqrt{1+|D u_j|^2}} \right) = k H_j, \quad \textmd{in } B_{r}\, ,
\end{equation}
$j=0,1$. It is well-known that (see \cite{GT})
$$
w=u_1-u_0
$$
satisfies a linear elliptic equation of the form
\begin{equation}\label{Lw = H1-H0}
Lw=k(H_1-H_0),
\end{equation}
where
\begin{equation}\label{Lw}
Lw = \sum_{i,j=1}^k \frac{\pa}{\pa x_j} \left( a^{ij}(x) \frac{\pa w}{\pa x_i} \right),
\end{equation}
with
\begin{equation*}
a^{ij}(x) = \ints_0^1 \frac{\pa}{\pa \zi_j} \left( \frac{\zi_i}{\sqrt{1+|\zi|^2}} \right)\bigg{|}_{\zi=D u_t(x)} dt,
\end{equation*}
and
\begin{equation*}
u_t(x)=tu_1(x) + (1-t)u_0(x), \quad x \in B_r.
\end{equation*}
From \eqref{ellipticity MCeq}, we find that
\begin{equation}\label{ellipticity Lw}
|\xi|^2 \leq a^{ij}(x) \xi_i \xi_j \leq |\xi|^2 \int_0^1(1+|D u_t(x)|^2) dt,
\end{equation}
where we used Einstein summation convention.
The following Harnack's type inequality will be one of the crucial tools for proving the stability result.

\begin{lemma} \label{lemma Harnack interior}
Let $u_j$, $j=0,1$, be two classical solutions of \eqref{MCeq uj}, with $u_1-u_0 \geq 0$ in $B_r$, and assume that
\begin{equation}\label{nabla uj leq M}
\| D u_j \|_{C^1(B_r)} \leq M, \quad j=0,1,
\end{equation}
for some positive constant $M$.
Then there exists a constant $K_1$, depending only on the dimension $k$ and $M$, such that
\begin{equation}\label{eq Harnack C1}
\|u_1-u_0\|_{C^1(B_{r/4})} \leq K_1 ( \inf_{B_{r/2}}(u_1-u_0) + \|H_1 - H_0\|_{C^0(B_r)}).
\end{equation}
\end{lemma}

\begin{proof}
We have already observed that $w=u_1-u_0$ satisfies \eqref{Lw = H1-H0} in $B_r$. From \eqref{ellipticity Lw} and \eqref{nabla uj leq M}, we find that $Lw$ is uniformly elliptic with continuous bounded coefficients, that is
\begin{equation*}
|\xi|^2 \leq a^{ij}(x) \xi_i \xi_j \leq |\xi|^2 (1+M^2),
\end{equation*}
and
\begin{equation*}
\Big{|} \frac{\pa }{\pa x_j} a^{ij}(x) \Big{|} \leq M',
\end{equation*}
for some positive $M'$ depending only on $M$. 

From Theorems 8.17 and 8.18 in \cite{GT}, we obtain the following Harnack's inequality
\begin{equation*}
\sup_{B_{r/2}} w \leq C_1 ( \inf_{B_{r/2}} w + \|H_1 - H_0\|_{C^0(B_r)}).
\end{equation*}
Then we use Theorem 8.32 in \cite{GT} and obtain that
\begin{equation*}
|w|_{C^{1,\alpha}(B_{r/4})} \leq C_2 \left(\|w\|_{C^0(B_{r/2})} + \|H_1 - H_0\|_{C^0(B_{r/2})} \right),
\end{equation*}
where $|\cdot|_{C^{1,\alpha}(B_{r/4})}$ is the $C^{1,\alpha}$ seminorm in $B_{r/4}$, with $\al \in (0,1)$.
By combining the last two inequalities, we obtain \eqref{eq Harnack C1} at once.
\end{proof}

Another crucial tool for our result is the following boundary Harnack's type inequality (or \emph{Carleson estimate}, see \cite{CS}).


\begin{lemma} \label{lemma Carleson}
Let $E$ be a domain in $\RR^k$ and let $T$ be an open set of $\pa E$ which is of class $C^2$. Let $u_j \in C^2(\overline{E})$, $j=0,1,$ be two solutions of
\begin{equation}\label{eq MC half a ball}
\diver \left( \frac{D u_j}{\sqrt{1+|D u_j|^2}} \right) = k H_j, \quad \textmd{in } E,
\end{equation}
with $j=0,1$, satisfying $\|D u_j\|_{C^1(E)} \leq M$ for some positive $M$.
Let $x_0 \in T$ and $r>0$ be such that $B_r(x_0) \cap \pa E \subset T$, and assume that
$$
u_1-u_0 \geq 0 \ \ \textmd{in } B_r(x_0) \cap E, \quad u_1 - u_0 \equiv 0 \ \ \textmd{on } B_r\cap \pa E .
$$
Assume further that $e_1$ is the interior normal to $E$ at $x_0$. Then, there exists a constant $K_2>0$ such that
\begin{equation}\label{boundary Harnack}
\sup_{B_{r/4}(x_0) \cap E} (u_1-u_0) \leq K_2 \left((u_1-u_0)\left(x_0 + \frac{r}{2}e_1\right)+ \|H_1-H_0\|_{C^0(B_r)}\right),
\end{equation}
where the constant $K_2$ depends only on the dimension $k$, $M$ and the $C^2$ regularity of $T$.
\end{lemma}

\begin{proof}
The proof is analogous to the one of Lemma \ref{lemma Harnack interior}, where we use Theorem 1.3 in \cite{BCN} and Corollary 8.36 in \cite{GT} in place of Theorems 8.17, 8.18 and 8.32 in \cite{GT}.
\end{proof}

We conclude this subsection with a quantitative version of the Hopf Lemma. We start with a statement which is valid for a general second order elliptic operator of the form
\begin{equation} \label{L cal}
\mathcal{L} w = \sum_{i,j=1}^k a^{ij} w_{x_i x_j} + \sum_{i=1}^k b^i w_{x_i},
\end{equation}
satisfying the ellipticity conditions
\begin{equation} \label{ellipt L cal}
 a^{ij} \zi_i \zi_j  \geq \lam |\zi|^2 , \quad \textmd{ and } \quad |a^{ij}|,|b^i| \leq \Lambda, \ \ i,j=1,\ldots,k,
\end{equation}
for $\lam, \Lambda >0$.

\begin{lemma} \label{lemma hopf migliorato}
Let $r>0$ and $\ga \geq 0$ be given. Assume that $w \in C^2 (B_r) \cap C^0(\overline{B}_r)$ fulfills the following conditions
\begin{equation*}
\mathcal{L} w \leq \ga \ \  \textmd{ and }  \ \ w\geq 0 \ \ \ \textmd{in } B_{r},
\end{equation*}
with $\mathcal{L}$ given by \eqref{L cal}.

Then, there exists a positive constant $C$ depending on $k, \lam, \Lambda,$ and upper bound on $\gamma$ such that for any $x_0 \in \pa B_r$ we have that 
\begin{equation}\label{hopf migliorato 1}
\sup_{B_{r/2}} w \leq C \left( \frac{w((1-t/r)x_0)}{t} + \ga \right), \quad \textmd{for any }\ 0< t \leq r/2.
\end{equation}
Moreover, if $w(x_0)=0$ then we have that 
\begin{equation}\label{hopf migliorato 2}
\sup_{B_{r/2}} w \leq C \left( \frac{ \pa w(x_0)}{\pa \nu} + \ga \right), 
\end{equation} 
where $\nu$ denotes the inward normal to $\pa B_r$.
\end{lemma}

\begin{proof}
In the annulus $A=B_r \setminus \overline{B}_{r/2}$,  we consider the auxiliary function
\begin{equation*}
v(x) =  \Big( \min_{\overline{B}_{r/2}} w \Big) \frac{e^{-\al|x|^2}-e^{-\al r^2}}{e^{-\al (r/2)^2} - e^{-\al r^2}}  + e^{\beta |x|^2} - e^{\beta r^2}, \end{equation*}
where 
$$
\al = \frac{(k + r \sqrt{k}) \Lambda}{2 \lam^2},  \quad \beta = \gamma \Big[k\lam - \sqrt{k}\Lambda r +  \sqrt{(k\lam - \sqrt{k}\Lambda r)^2 + \gamma \lam r^2}  \Big]^{-1}.
$$
Here, the constants $\al $ and $\beta$ are chosen in such a way that $v$ satisfies
$$
\mathcal{L} v \geq \gamma.
$$
We notice that
\begin{equation}\label{eq 1 hopf migl}
\frac{v((1-t/r)x_0)}{t} \geq \frac{\al r e^{-\al r^2}}{e^{-\al (r/2)^2} - e^{-\al r^2}}  \Big( \min_{\overline{B}_{r/2}} w \Big) - 2\beta r e^{\beta r^2}.
\end{equation}
Since $v=0$ on $\pa B_r$ and $v \leq \min_{\pa B_{r/2}} w$ on $\pa B_{r/2}$, we have that the function $w-v$ satisfies the following conditions
\begin{equation*}
\begin{cases}
\mathcal{L} (w-v) \leq 0, & \textmd{in } A,\\
w-v \geq 0, & \textmd{on } \pa A.
\end{cases}
\end{equation*}
Hence, by maximum principle we have that $w-v \geq 0$ in $\overline{A}$, and from \eqref{eq 1 hopf migl} we obtain that
\begin{equation} \label{eq 2 hopf migl}
\min_{\overline{B}_{r/2}} w \leq \frac{e^{3\al r^2/4} -1}{\al r} \left(\frac{w((1-t/r)x_0)}{t} + 2\beta r e^{\beta r^2} \right),
\end{equation}
for $0<t<r/2$. As in the proof of Lemma \ref{lemma Harnack interior}, we use Theorems 8.17 and 8.18 in \cite{GT} to get
\begin{equation*}
\max_{\overline{B}_{r/2}} w \leq C_1 ( \min_{\overline{B}_{r/2}} w +\gamma),
\end{equation*}
and from \eqref{eq 2 hopf migl} we obtain \eqref{hopf migliorato 1} and \eqref{hopf migliorato 2}.
\end{proof}

We will use Lemma \ref{lemma hopf migliorato} in the following form.

\begin{lemma} \label{lemma hopf}
Let $E,\ T,\ u_0,\ u_1,\ M,$ and $x_0$ be as in Lemma $\ref{lemma Carleson}$, with 
\begin{equation*}
u_1-u_0 \geq 0 \quad \textmd{ in } E.
\end{equation*} 
Assume that there exists $B_r(c) \subset E$ with $x_0 \in \pa B_r(c) \cap T$.
Let
\begin{equation*}
\ell = \frac{c-x_0}{r}.
\end{equation*}
Then, there exists a constant $K_3$ such that
\begin{equation} \label{Hopf A}
\|u_1-u_0\|_{C^1(B_{r/4}(c))} \leq K_3 \left( \frac{(u_1-u_0)(x_0+t \ell)}{t} + \|H_1-H_0\|_{C^0(B_r(c))} \right),
\end{equation}
for every $t \in (0,r/2)$, and
\begin{equation} \label{Hopf B}
\|u_1-u_0\|_{C^1(B_{r/4}(c))} \leq K_3 \left( \frac{\pa (u_1-u_0)}{\pa \ell}(x_0) + \|H_1-H_0\|_{C^0(B_r(c))} \right),
\end{equation}
for $t=0$. The constant $K_3$ depends only on the dimension $k$, on $\ M,$ and $\rho$, and upper bound on $\|H_1-H_0\|_{C^0(B_r(c))}$.
\end{lemma}

\begin{proof} 
As we have shown in the proof of Lemma \ref{lemma Harnack interior}, $w=u_1-u_0$ satisfies \eqref{Lw = H1-H0}, which is uniformly elliptic. Moreover, we notice that, by letting
$$\ga = \|H_1-H_0\|_{C^0(B_r(c))},$$
we have that
\begin{equation*}
Lw \leq \ga.
\end{equation*}
Hence, we can apply Lemma \ref{lemma hopf migliorato} and, by using Lemma \ref{lemma Harnack interior}, we conclude.
\end{proof}

\subsection{The symmetry result of Alexandrov}\label{subsection Alexandrov}
In order to make the paper self-contained, we give a sketch of the proof of the Soap Bubble Theorem by Alexandrov. This will be the occasion to set up some necessary notation.

Let $S$ be a $C^2$ regular, connected, closed hypersurface embedded in $\RR^{n+1}$, $n\geq 1$,  and let  $\Omega$ be the relatively compact domain of $\mathbb R^{n+1}$ bounded by $S$. 
Let $\om\in\RR^{n+1}$ be a unit vector and $\lam\in\RR$ be a parameter. For an arbitrary set $A$, we define the following objects:
\begin{equation}
\label{definitions}
\begin{array}{lll}
&\pi_{\lam}=\{ \xi \in \RR^{n+1}: \xi \cdot\om=\lam\}\ &\mbox{a hyperplane orthogonal to $\om,$}\\
&A^{\lam}=\{p\in A: p\cdot\om>\lam\}\ &\mbox{the right-hand cap of $A$},\\
&\xi^{\lam}=\xi-2(\xi \cdot\om-\lam)\,\om\ &\mbox{the reflection of $\xi$ about $\pi_{\lam},$}\\
&A_{\lam}=\{p\in\RR^{n+1}\,:\,p^{\lam}\in A^{\lam}\}\ &\mbox{the reflected cap about $\pi_{\lam}$},\\
&{\hat A}_\lam=\{p\in A :  p \cdot \om < \lam\} &\mbox{the portion of $A$ in the left-hand half plane.}
\end{array}
\end{equation}
Set $\mathcal{M}=\max\{p\cdot\om: p\in S \}$, the extent of $S$ in the direction $\om$; if $\lam<\mathcal{M}$ is close to $\mathcal{M}$, the reflected cap $\Om_\lam$ is contained in $\Om$. Set
\begin{equation}\label{m def}
m=\inf\{\mu: \Om_{\lam}\subset \Om \mbox{ for all } \lam\in(\mu,\mathcal{M})\}.
\end{equation}
Then for $\lambda = m$ at least one of the following two cases occurs:
 \begin{enumerate}
\item[(i)]
$S_m$ becomes internally tangent to $S$ at some
point $p\in S\setminus\pi_m;$
\item[(ii)]
$\pi_m$ is orthogonal to $S$ at some point $p\in S\cap\pi_m.$
\end{enumerate}

\begin{thmx}[Alexandrov Soup Bubble Theorem]
Let $S$ be a $C^2$-regular, closed, connected hypersurface embedded in  $\RR^{n+1}$. If the mean curvature $H$ of $S$ is constant, then $S$ is a sphere.
\end{thmx}

\begin{proof}
Let $\om$ be a fixed direction. We apply the method of moving planes in the direction $\om$ and we find a critical position for $\lam = m$.

If Case (i) occurs, then we locally write $S_m$ and $S$ as graphs of function $u_1$ and $u_0$, respectively, over $B_r \cap T_p S$ (which coincides with $T_p S_m$), where $p$ is the tangency point. It is clear that $w = u_1-u_0$ is non-negative and, since $H$ is constant, we have that $w$ satisfies
\begin{equation*}
Lw=0, \quad  \textmd{in } \ B_r \cap T_p S,
\end{equation*}
for some $r>0$, and where $L$ is given by \eqref{Lw}. Since $w(0)=0$, by the strong maximum principle we obtain that $w \equiv 0$ in $ B_r \cap T_p S$, that is $S$ and $S_m$ coincides in an open neighborhood of $p$.

If Case (ii) occurs, then we locally write $S_m$ and $S$ as graphs of function $u_1$ and $u_0$, respectively, over $T_p S \cap \{x \cdot \om \leq m\}$. As for case (i), we find that there exists $r>0$ such that
\begin{equation*}
\begin{cases}
Lw = 0, & \textmd{in } \ B_r \cap T_p S \cap \{x\cdot \om < m\}, \\
w=0, & \textmd{on } \ B_r \cap T_pS \cap \{x\cdot \om = m\}.
\end{cases}
\end{equation*}
Since $ \nabla w (0) = 0$ 
and from the Hopf Lemma (see for instance \cite{GT}) we obtain that $w\equiv 0$ in $B_r \cap T_pS \cap \{x\cdot \om \leq m\}$.

Hence, in both cases (i) and (ii) we have that the set of tangency points (that is those points for which case (i) or (ii) occur) is open. Since it is also closed and non-empty we must have that $S_m=\hat{S}_m$, that is $S$ is symmetric about the hyperplane $\pi_m$. Since $\om$ is arbitrary, we find that $S$ is symmetric in every direction.

Up to a translation, we can assume that the origin $O$ is the center of mass of $S$. Since $O$ belongs to every axis of symmetry and every rotation can be written as a composition of reflections, we have that $S$ is invariant under rotations, which implies that $S$ is a sphere.
\end{proof}

\subsection{Curvatures of projected surfaces} \label{subsect Luigi}
Before giving the results of this subsection, we need to recall some basic facts about hypersurfaces in $\mathbb R^{n+1}$, in particular about the interplay between the normal and the principal curvatures.
 Let $U$ be an orientable hypersurface of class $C^2$ embedded in $\mathbb R^{n+1}$ (which in the proof of Theorem \ref{main} will be an open set of the surface $S$). The choice of an orientation on $U$ is equivalent to the choice of a Gauss map $\nu\colon U\to \mathbb S^n$ (in this general context there is no a canonical orientation). Fixed a point $q\in U$, we denote by $W_q\colon T_qU\to T_qU$ the shape operator $W_q=-d\nu_q$. $W_q$ is a symmetric operator and its eigenvalues $\kappa_i(q)$ are the principal curvatures of $U$ at $q$. We assume that $\kappa_1(q)\leq \kappa_2(q)\leq\dots \leq \kappa_n(q)$. The first and the last principal curvature can be obtained as minimum and maximum of the normal curvature. Here we recall that, given a non-zero vector $v\in T_qU$, its normal curvature $\kappa(q,v)$ is defined as
$$
\kappa(q,v)=\frac{1}{|v|^2}\, W_q(v)\cdot v\,.
$$
$\kappa(q,v)$ can be alternatively written in term of curves as
$$
\kappa(q,v)=\frac{1}{|\dot\alpha(0)|^2}\,  \nu_{\alpha(0)}\cdot \ddot\alpha(0)
$$
where $\alpha\colon I\to U$ is an arbitrary curve satisfying $\alpha(0)=0$, $\dot\alpha(0)=v$.

In order to perform a quantitive study of the moving planes, we need to handle the following situation: given a hypersurface $U$  of class $C^2$ in $\RR^{n+1}$, we consider its intersection $U'$ with an affine hyperplane $\pi_1$ (in the proof of Theorem \ref{thm approx symmetry 1 direction} $\pi_1$ will be the critical hyperplane in the direction $\om$). If $\pi_1$ intersects $U$ transversally, $U'=U\cap \pi_1$ is a hypersurface of class $C^2$ of $\pi_1$ and we consider its projection $U''$ onto another hyperplane $\pi_2$ of $\RR^{n+1}$ (which will be the tangent hyperplane to the reflected cap at some point which is close to the critical hyperplane). An example in $\RR^3$ is shown in Figure \ref{figura1}.
The next  two propositions allow us to control the principal curvature of $U''$ in terms of the principal curvature of $U$ and the normal vectors to $\omega_1$ and $\omega_2$. 

\begin{figure}[h]
\centering
 \includegraphics[width=0.4\textwidth]{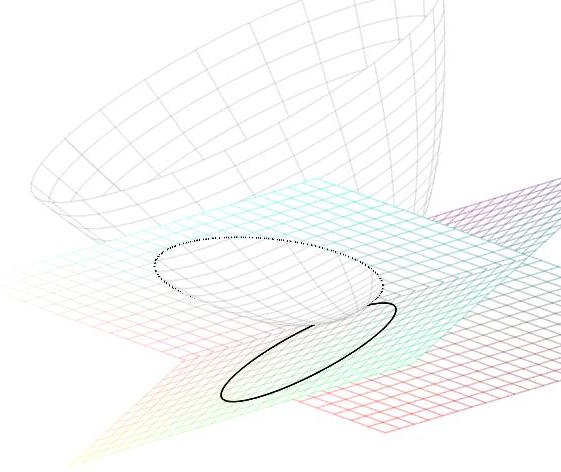}
  \caption{In the figure $U$ is the parabololid $z=x^2+y^2$, $\pi_1$ is the affine plane $z=2+8y$ and $\pi_2$ is the plane $z=0$. In this case $U'$ is the dashed ellipse in $\pi_1$, while $U''$ is the circle projected in $\pi_2$.} \label{figura1}
\end{figure}

\begin{proposition}\label{prop Luigi I}
 Let $U$ be an orientable hypersurface of class $C^2$ embedded in $\mathbb R^{n+1}$ with principal curvatures $\kappa_j$, $j=1,\ldots,n$, and Gauss map $\nu$. 
Let $\pi$ be an hyperplane of $\RR^{n+1}$ intersecting $U$ transversally and let $U'=U\cap \pi$. Then $U'$ is  an orientable hypersurface of class $C^2$ embedded in $\pi$ and, once a Gauss map $\nu'\colon U'\to \mathbb S^{n-1}$ is fixed, its principal curvatures $\kappa_i'$ satisfy 
\begin{equation} \label{bound curv I}
\frac{1}{ \nu_q\cdot \nu'_q}\kappa_1(q)\leq \kappa'_i(q)\leq \frac{1}{ \nu_q\cdot \nu'_q }\kappa_{n}(q) \,.
\end{equation}
for every $q\in U'$ and $i=1,\dots,n-1$.
\end{proposition}
\begin{proof}
First of all we observe that  $U'$ is of class $C^2$ by the implicit function theorem and it is orientable since the map $\nu'\colon U'\to \mathbb S^{n-1}$
defined by 
\begin{equation} \label{nu primo}
\nu'_q=-*(*(\nu_q\wedge \omega)\wedge \omega),
\end{equation}
is a Gauss map on $U'$, where by $*$ we denote the Hodge \lq\lq star\rq\rq operator in $\RR^{n+1}$ computed with respect to the standard metric and the standard orientation.

In order to prove \eqref{bound curv I}: fix $q\in U'$ and consider an  arbitrary unitary vector 
$v\in T_qU'$. Let $\kappa(q,v)$ be normal curvature of $(q,v)$ in $U$. Then
$$
\kappa(q,v)=\nu_q\cdot \ddot \alpha(0)
$$
where $\alpha$ is an arbitrary smooth curve in $U'$ parametrized by arc length and such that  $\alpha(0)=q$ and $\dot \alpha(0)=v$. Since $\nu_q$ is orthogonal to $T_qU'$,  it belongs to the plane generated by $\omega$ and $\nu'_q$ and we can write
$$
\nu_q=(\nu'_q\cdot \omega)\, \omega+( \nu_q\cdot \nu'_q) \,\nu'_q\,.
$$
Therefore
$$
\kappa(q,v)=(\nu_q\cdot \ddot \alpha(0))=(\nu_q\cdot \nu'_q)( \nu'_q\cdot \ddot \alpha(0)) =(\nu_q\cdot \nu'_q) \kappa'(q,v),
$$
where $\kappa'(q,v)$ is the normal curvature of $U'$ in $(q,v)$, and the claim follows.
\end{proof}

Note that in Proposition \ref{prop Luigi I}, if we chose as Gauss map the one defined as in \eqref{nu primo} we have  
$$
(\nu_q\cdot \nu'_q)=-*(*(\nu_q\wedge \omega)\wedge \omega)\cdot\nu_q=- *(\nu_q\wedge \omega)\wedge \omega\cdot*\nu_q\,.
$$
Now we choose the positive oriented orthonormal basis of $\mathbb R^{n+1}$
$\{e_1,\dots,e_{n},\nu_q\}$ where the first $n$-vectors are an othonormal basis of $T_{q}U$. Then we have
$$
*(\nu_q\wedge \omega)=-(\omega\cdot e_{n}) \,e_1\wedge \dots \wedge e_{n-1}\,,\quad  *\nu_q= e_1\wedge \dots \wedge e_{n},
$$
and
$$
 \nu_q\cdot \nu'_q=(\omega\cdot e_{n})  ( e_1\wedge \dots \wedge e_{n-1}\wedge \omega)\cdot ( e_1\wedge \dots \wedge e_{n})=( \omega\cdot e_{n})^2=1-(\nu_q\cdot\omega)^2 \,,
$$
i.e. 
\begin{equation}\label{prodottodeinormali}
\nu_q\cdot \nu'_q=1-(\nu_q\cdot\omega)^2.
\end{equation}
Therefore, when $\nu_q'$ is given by \eqref{nu primo}, \eqref{bound curv I} reads as
\begin{equation} \label{bound curv Iuigi}
\frac{1}{1-(\nu_q\cdot\omega)^2  }\kappa_1(q)\leq \kappa'_i(q)\leq \frac{1}{ 1-(\nu_q\cdot\omega)^2  }\kappa_{n}(q) \,.
\end{equation}
for $i=1,\dots,n-1$.

\begin{proposition} \label{prop Luigi II}
Let $\omega_1$ and $\omega_2$ be unit vectors in  $\mathbb R^{n+1}$, denote by $\pi_1$ an hyperplane orthogonal to $\omega_1$, and let 
$\pi_2$ be the hyperplane  orthogonal to $\omega_2$ passing through the origin of $\RR^{n+1}$. Let $U'$ be a $C^2$ regular oriented hypersurface of $\pi_1$ such that $\omega_2$ is not tangent to $U'$ at any point. Denote by $\kappa'_i$, for $i=1,\dots, n-1$,  the principal curvatures of $U'$ and denote by $\nu'$ the normal vector to $U'$.
Then the orthogonal projection $U''$ of  $U'$ onto $\pi_2$ is a $C^2$-regular hypersurface of $\pi_2$ with a canonical orientation. 
Moreover, for any $q\in U'$ we have 
\begin{equation}\label{bound curvatures II}
|\kappa_i''({\rm pr} (q))| \leq \frac{ |\om_1\cdot  \om_2|  }{\left[(\omega_1\cdot \omega_2)^2+(\omega_2\cdot \nu_q')^2\right]^{\frac{3}{2}}} \max\{|\kappa_1'(q)|,|\kappa_{n-1}'(q)|\},
\end{equation}
for every $i=1,\dots,n-1$, 
where ${\rm pr} (q)$ is the projection of $q$ onto $\pi_2$ and $\{\kappa''_i\}$ are the principal curvature of $U''$. 
\end{proposition}
\begin{proof}
If $X$ is a local positive oriented parametrisation of $U'$, then $Y=X-(X\cdot \omega_2)\omega_2$ is a local parametrisation of $U''$, and 
$$
\nu''\circ Y:={\rm vers}(*(Y_1\wedge Y_2\wedge \dots \wedge Y_{n-1}\wedge \omega_2))
$$
defines a Gauss map for $U''$, where $Y_k$ is the $k^{th}$ derivative of $Y$ with respect to the coordinates of its domain. Therefore $U''$ is a $C^2$-regular hypersurface of $\pi_2$ oriented by the map $\nu''$. 

Now we prove inequalities \eqref{bound curvatures II}. Fix a point $q\in U'$ and let ${\rm pr} (q)=q-( q\cdot \omega_2) 	\,\omega_2$ be its projection onto $U''$.
Let $X$ be a local positive oriented parametrization of $U'$ around $q$ and $Y=X-(X\cdot\omega_2)\omega_2$ be the induced parametrization of $U''$ around ${\rm pr}(q)$.  
Let $\beta \colon (-\delta,\delta)\to U''$ be an arbitrary regular curve contained in $U''$ such that  $\beta(0)={\rm pr} (q)$ and let
$$
v=\frac{\dot \beta(0)}{|\dot \beta(0)|}\,,\quad g=\frac{1}{|\dot \beta |^2}\nu''_\beta\cdot \ddot\beta\,.
$$
Then 
$$
g(0)=\kappa''({\rm pr}(q),v)\,,
$$
where $\kappa''({\rm pr}(q),v)$ is the normal curvature of $U''$ at $(q,v)$. 
The curve $\beta$ can be seen as the projection of a regular curve $\alpha$ in $U'$ passing through $p$. Since
$\nu''_\beta$ is orthogonal to $\omega_2$ we have
 $$
g=\frac{1}{|\dot \beta|^2}\nu''_\beta\cdot \ddot\alpha\,.
$$
Note that, since
$$
Y_k=X_k-( X_k\cdot\omega_2) \omega_2\,,
$$
then we have
$$
\nu''\circ Y={\rm vers}(*(X_1\wedge X_2\wedge \dots \wedge X_{n-1}\wedge \omega_2))
$$
and
$$
g=\frac{(*(X_1(\tilde \alpha)\wedge \dots \wedge X_{n-1}(\tilde \alpha)\wedge \omega_2)\cdot \ddot\alpha}{|\dot \beta|^2 | X_1(\tilde \alpha)\wedge \dots \wedge X_{n-1}(\tilde \alpha)\wedge \omega_2)|}\,.
$$
Now, it is simply to prove that
$$
(*(X_1(\tilde \alpha)\wedge \dots \wedge X_{n-1}(\tilde \alpha)\wedge \omega_2)\cdot \ddot\alpha=
(\omega_1\cdot\omega_2)*(X_1(\tilde \alpha)\wedge \dots \wedge X_{n-1}(\tilde \alpha)\wedge \omega_1)\cdot \ddot\alpha\,,
$$
and therefore
$$
g=\frac{\omega_1\cdot\omega_2}{|\dot \beta|^2 }\,\,\frac{(*(X_1(\tilde \alpha)\wedge \dots \wedge X_{n-1}(\tilde \alpha)\wedge \omega_1))\cdot \ddot\alpha}{ |X_1(\tilde \alpha)\wedge \dots \wedge X_{n-1}(\tilde \alpha)\wedge \omega_2)|}\,,
$$
%
which implies
$$
g= (\nu'_\alpha\cdot\ddot\alpha)\frac{\omega_1\cdot\omega_2}{|\dot \beta|^2 }\,\,\frac{|X_1(\tilde \alpha)\wedge \dots \wedge X_{n-1}(\tilde \alpha)\wedge \omega_1|}{ |X_1(\tilde \alpha)\wedge \dots \wedge X_{n-1}(\tilde \alpha)\wedge \omega_2|}\,.
$$
%
We may assume that $\alpha$ is parametrised by arc length and so
$$
|\dot \beta|^2=1-(\dot \alpha\cdot \omega_2)^2\,,
$$
which implies
$$
g=(\nu'_\alpha\cdot\ddot\alpha)\frac{ \omega_1\cdot \omega_2}{1-( \dot \alpha\cdot \omega_2)^2}
\frac{|X_1(\tilde \alpha)\wedge \dots \wedge X_{n-1}(\tilde \alpha)\wedge \omega_1|}{ |X_1(\tilde \alpha)\wedge \dots \wedge X_{n-1}(\tilde \alpha)\wedge \omega_2|}\,.
$$
Moreover a standard computation yields that 
$$
\frac{|X_1(\tilde \alpha)\wedge \dots \wedge X_{n-1}(\tilde \alpha)\wedge \omega_1|}{ |X_1(\tilde \alpha)\wedge \dots \wedge X_{n-1}(\tilde \alpha)\wedge \omega_2|}=\frac{1}{(1-| \omega_2-(\omega_2\cdot \nu'_\alpha)\nu'_\alpha|^2)^{1/2}}\,,
$$
and hence
$$
g(0)=\kappa'(q,\dot\alpha(0))
\frac{ \omega_1\cdot \omega_2}{(1-| \omega_2-(\omega_2\cdot \nu'_q)\nu'_q|^2)^{1/2}}\,\frac{1}{1-( \dot \alpha(0)\cdot \omega_2)^2}\,,
$$
where $\kappa'(q,\dot\alpha(0))$ is the normal curvature of $U'$ at $(q,\dot\alpha(0))$. Therefore 
$$
\kappa''({\rm pr}(q),v)=\kappa'(q,\dot\alpha(0))
\frac{ \omega_1\cdot \omega_2}{(1-| \omega_2-(\omega_2\cdot \nu'_q)\nu'_q|^2)^{1/2}}\,\frac{1}{1-( \dot \alpha(0)\cdot \omega_2)^2}\,.
$$
In particular 
\begin{equation}\label{k1}
\kappa''_1({\rm pr} (q))=\kappa'_1(q)
\frac{ \omega_1\cdot\omega_2}{(1-| \omega_2-(\omega_2\cdot \nu'_q)\nu'_q|^2)^{1/2}}\inf_{v\in \mathbb S^{n-1}_q} \frac{1}{1-( v\cdot \omega_2)^2}\,,
\end{equation}
and
\begin{equation}\label{k2}
 \kappa''_{n-1}({\rm pr} (q))= \kappa'_{n-1}(p)\frac{\omega_1\cdot\omega_2}{(1-| \omega_2-(\omega_2\cdot \nu'_q)\nu'_q|^2)^{1/2}}\sup_{v\in  \mathbb S^{n-1}_q} \frac{1}{1-( v\cdot \omega_2)^2}\,,
\end{equation}
where $ \mathbb S^{n-1}_q=\{v\in T_qU'\,\,:\,\,|v|=1\}$. 
Now if $v\in  \mathbb S^{n-1}_q$, we have 
$$
1-( v\cdot \omega_2)^2\leq 1-| \omega_2-(\omega_2\cdot \nu'_q)\nu'_q|^2\,.
$$
Therefore  
$$
|\kappa_i''({\rm pr} (q))| \leq \frac{ |\om_1\cdot  \om_2|  }{(1-| \omega_2-(\omega_2\cdot \nu'_q)\nu'_q|^2)^{3/2}} \max\{|\kappa_1'(q)|,|\kappa_{n-1}'(q)|\},
$$
for every $i=1,\dots,n-1$. Finally, since $\RR^{n+1}=T_qU'\oplus \langle \omega_1\rangle \oplus \langle \nu'_q\rangle$, we have
$$
1-| \omega_2-(\omega_2\cdot \nu'_q)\nu'_q|^2=(\omega_1\cdot\omega_2)^2+(\omega_2\cdot \nu_q')^2
$$
and the claim follows. 
%
\end{proof}

\section{Technical Lemmas} \label{section technical lemmas}
Let $S$ be a connected closed $C^2$ regular hypersurface embedded in $\RR^{n+1}$ and let $\rho$ be the radius of the uniform touching sphere.
%
%

%

Let $S_m$ and $\pi_m$ be as in \eqref{definitions} and let $\pa S_m=S \cap \pi_m$. It will be useful to define the following set
\begin{equation}\label{Gamma m t}
S_m^\delta = \{p\in S_m:\ d_S(p,\pa S_m) > \delta \},
\end{equation}
for $\de>0$.

\begin{lemma} \label{lemma Ga t}
Let $0<\de<\rho$ and set $ \si = \rho \sin(\de /\rho)$. Then the following facts hold:
\begin{itemize}
\item[(i)] For any $p \in S_m^\de$ we have 
$\mathcal{U}_{\sigma}(p) \subset S_m.$

\vspace{0.1cm}
\item[(ii)] For any $q \in S_m \setminus S_m^\de$ there exists $p \in \pa S_m$ and $x\in B_{\de}\cap T_pS$ such that
\begin{equation*}
q=p+x+ u(x) \nu_p.
\end{equation*}
\end{itemize}
Here $u$ and $\mathcal U$ are as in \eqref{urp}.
\end{lemma}

\begin{proof}
\begin{enumerate}
\item[(i)] Let $x \in B_\sigma\cap T_pS$ and let $q=p+x+ u(x) \nu_p$. Since
\begin{equation*}
d_S(q,\pa S_m) \geq d_S(p,\pa S_m) - d_S(p,q),
\end{equation*}
\eqref{bound on d Ga above} implies
\begin{equation*}
d_S(q,\pa S_m) \geq \de - \rho \arcsin \frac{|x|}{\rho}.
\end{equation*}
The assumption $|x|< \si$ implies the thesis.

\vspace{0.3cm}
\item[(ii)] Let $p \in \pa S_m$ be such that $d_S(q,\pa S_m) = d_S(p,q)$, and let $x$ be the orthogonal projection of $q$ onto $T_pS$. Since $|x| \leq d_S(p,q)<\de$ and $\de<\rho$, then $|x|<\rho$ and Lemma \ref{lemma bounds rho} implies the statement.
\end{enumerate}
\end{proof}

In the next lemma we show that any two points in $S_m^\de$ can be joined by a piecewise geodesic curve and we give a bound on its length. An analogous lemma was proved in \cite{ABR} in the special case when $S_m^\de$ is contained in a hyperplane.

\begin{lemma} \label{lemma piecewise geodesic}
Let $0< \de <\rho$, and set
\begin{equation}\label{L number def}
L= \frac{|S| 2^n}{\omega_n  \de^{n}}
\end{equation}
where $\omega_n$ is volume of  the unit ball in $\RR^n$.  
Let $p,q$ be in a connected component of $S_m^\de$. Then there exists a piecewise geodesic path $\ga:[0,1] \to S_m^{\de/2}$ satisfying $\ga(0)=p$ and $\ga(1)=q$ and with length bounded by $L$. Moreover, $\ga$ can be built by joining $\mathcal{N}$ minimal geodesics of length $\de$, with
\begin{equation}\label{q leq L}
\mathcal{N} \leq L,
\end{equation}
and one minimal geodesic of length less or equal than $\de$.
\end{lemma}

\begin{proof}
Let $p,q$ be in a connected component of $S_m^\de$.
We can join $p$ and $q$ by a path $\tilde{\ga}:[0,1] \to S_m^\de$ such that $\tilde{\ga}(0)=p$ and $\tilde{\ga}(1)=q$. Given a point $z_0 \in S$, we denote by $D_r(z_0)$ the set of points on $S$ with intrinsic distance from $z_0$ less than $r$, i.e.
\begin{equation*}
D_r(z_0)= \{z \in S:\ d_{S}(z,z_0) < r \}\,.
\end{equation*}
When $r<\rho$, \eqref{bound on d Ga above} implies
\begin{equation}\label{D r  geq}
|D_r(p)| \geq \om_n r^n.
\end{equation}

Then we consider the increasing sequence $\{t_0, t_1, \ldots, t_I\}$ in $[0,1]$ recursively defined as follows:
 $t_0=0$, and 
\begin{equation} \label{t i+1 def}
t_{i+1} =\inf\,\left\{t\in[0,1]\,\,:\,\, D_{\de/2}(\tilde{\ga}(s)) \cap \bigcup_{j=0}^i D_{\de/2}(\tilde{\ga}(t_j)) = \emptyset \, , \  \forall s \in [t,1]\right\}
\end{equation}
if
$$
\left\{t\in[0,1]\,\,:\,\, D_{\de/2}(\tilde{\ga}(s)) \cap \bigcup_{j=0}^i D_{\de/2}(\tilde{\ga}(t_j)) = \emptyset \ \forall s \in [t,1]\right\}\mbox{ is non-empty},
$$
and $t_{i+1}=t_I$, otherwise.
Therefore $\{t_0, t_1, \ldots, t_I\}$ is an increasing sequence in $[0,1]$ satisfying
\begin{equation} \label{D de i j disjoint}
D_{\de/2}(\tilde{\ga}(t_i)) \cap D_{\de/2}(\tilde{\ga}(t_j)) = \emptyset, \quad \textmd{for } i\neq j,\ i,j=0,\ldots,I,
\end{equation}
and
\begin{equation*}
D_{\de/2}(\tilde{\ga}(t_i)) \subset S_m^{\de/2}, \ \ i=0,\ldots,I.
\end{equation*}
We complete the sequence by adding $t_{I+1}=1$ as last term.
Since
\begin{equation*}
\Big{|} \bigcup_{i=0}^I D_{\de/2}(\tilde{\ga}(t_i)) \Big{|} \leq |S|,
\end{equation*}
from \eqref{D r  geq} and \eqref{D de i j disjoint} we obtain that
\begin{equation}\label{I leq}
I+1 \leq \frac{2^n}{\om_n \de^n} |S|.
\end{equation}
From \eqref{t i+1 def}, it is clear that
\begin{equation*}
\overline{D}_{\de/2}(\tilde{\ga}(t_i)) \cap \left( \bigcup_{j=0}^{i-1} \overline{D}_{\de/2}(\tilde{\ga}(t_j)) \right) \neq \emptyset,
\end{equation*}
for every $i=1,\ldots,I$. Let
\begin{equation*}
\sigma (i) = \max \{ j>i: \overline{D}_{\de/2}(\tilde{\ga}(t_i)) \cap  \overline{D}_{\de/2}(\tilde{\ga}(t_j)) \neq \emptyset \}\,.
\end{equation*}
Then we set $\sigma^2(i)=\sigma(\sigma(i))$, $\sigma^3(i)=\sigma(\sigma(\sigma(i)))$ and so on, and fix
 $\tau\in \NN$ such that
$\sigma^{\tau}(0)=I$.
We define $\ga_1$ as a minimal geodesic joining $p$ and $\tilde{\ga}(t_{\sigma(0)})$ such that 
\begin{equation*}
\ga_1 \subset \overline{D}_{\de/2}(p) \cup \overline{D}_{\de/2}(\tilde{\ga}(t_{\sigma(0)}));
\end{equation*}
for $i=2,\ldots,\tau$, we let $\ga_i$ be a minimal geodesic joining $\tilde{\ga}(t_{\sigma^i(0)})$ and $\tilde{\ga}(t_{\sigma^{i+1}(0)})$ and such that
\begin{equation*}
\ga_i \subset \overline{D}_{\de/2}(\tilde{\ga}(t_{\sigma^i(0)})) \cup \overline{D}_{\de/2}(\tilde{\ga}(t_{\sigma^{i+1}(0)})).
\end{equation*}
Moreover, we let $\ga_{\tau+1}$ be a minimal geodesic joining $\tilde{\ga}(t_{I})$ and $q$ and such that
\begin{equation*}
\ga_i \subset \overline{D}_{\de/2}(\tilde{\ga}(t_{\sigma^i(0)})) \cup \overline{D}_{\de/2}(q).
\end{equation*}
Let $\ga$ be the piecewise geodesic obtained by the union of $\ga_1,\ldots,\ga_{\tau+1}$. It is clear that each $\ga_i$ has length equal to $\de$ for $i=1,\ldots,\tau$, and less or equal than $\de$ for $i=\tau+1$. Since $\tau\leq I$, from \eqref{I leq} we obtain that
\begin{equation*}
length(\ga) \leq (\tau+1) \de \leq \frac{2^n}{\om_n \de^n} |S|,
\end{equation*}
which implies \eqref{L number def} and \eqref{q leq L}, and the proof is complete.
\end{proof}

%
It will be useful to define the following two numbers:
\begin{equation}\label{epsilon 0}
\ep_0=\min \left(\frac 1 2 , \frac{\rho}{16L} \sin \frac{\de}{2\rho} \right),
\end{equation}
and
\begin{equation}\label{n 0}
N_0= 1 + \left[ \log_{(1-\ep_0)} \frac{1}{2} \right] ,
\end{equation}
where $L$ is given by \eqref{L number def} and $[\cdot]$ is the integer part function. We have the following lemma.

\begin{lemma} \label{lemma harnack chain}
Let $\de \in (0,\rho)$, $\ep \in (0,\ep_0)$, with $\ep_0$ given by \eqref{epsilon 0}, and set
\begin{equation}\label{r_i}
r_i=(1-\ep)^i \rho \sin\frac{\de}{2\rho},
\end{equation}
for $i \in \NN$. Let $p$ and $q$ be any two points in a connected component of $S_m^\de$ . Then
there exist an integer $N\leq N_0$, with $N_0$ given by \eqref{n 0}, and a sequence of points  $\{p_1,\dots, p_N\}$ in  $S_m^{\de/2}$ 
such that
%
\begin{eqnarray}
&& p,q \in \bigcup_{i=0}^n \overline{\mathcal U}_{r_i/4}(p_i);,\\
&& \mathcal U_{r_0}(p_i) \subset S_m, \quad i=0,\ldots,N, \\
\label{Vr0 subset}&& p_{i+1} \in\overline{\mathcal U_{r_i/4}(p_i)}, \quad i=0,\ldots,N-1,
\end{eqnarray}
where $\mathcal{U}_{r_i}(p_i)$ are defined as in \eqref{urp}.
\end{lemma}

\begin{proof}
Let $\ga$ be a path as in Lemma \ref{lemma piecewise geodesic} and denote by $s$ its arclength.
Set $p_0=p$ and define $p_i = \ga (r_i/4)$, for each $i=1,\ldots,N-1$, and $p_N=q$. Here, $N$ is the largest integer such that
\begin{equation*}
\sum_{i=0}^{N-1} \frac{r_i}{4} \leq L.
\end{equation*}
Since $\ep < \ep_0$, we have
\begin{equation*}
\sum_{i=0}^{N_0-1} \frac{r_i}{4} > 2L,
\end{equation*}
and hence such $N$ exists and we can assume that $N\leq N_0$, where $N_0$ is defined by \eqref{n 0}.
Since $\ga \subset S_m^{\de/2}$, the assertion of the theorem easily follows from \eqref{bound on d Ga above}.
\end{proof}

For a fixed direction $\ell \in \Sbb^n$, we denote by $\ell^\perp$ the orthogonal subspace to $\ell$, i.e.
\begin{equation*}
\ell^\perp = \{z \in \RR^{n+1}:\ z \cdot \om  = 0\}.
\end{equation*}

\begin{lemma} \label{lemma change normal}
Let $p \in S$ and $u\colon B_r\cap T_pS\to \mathbb R$ be a $C^2$ map as in \eqref{urp}, with $r <  \rho$.
Let $\ell \in \Sbb^n$ be such that
\begin{equation}\label{omega assumptions}
\nu_p\cdot \ell >0 \quad \textmd{and }\ \ \ |\ell - \nu_p|<\ep,
\end{equation}
for some $0\leq \ep < 1$.
%
%
There exists a $C^2$ function $v: B_{r\sqrt{1-\ep^2}}\cap \ell^{\perp} \to \RR$ such that the set
\begin{equation}\label{V change normals}
V = \{p+ y + v(y) \ell : \ y \in  B_{r\sqrt{1-\ep^2}}\cap \ell^\perp\}
\end{equation}
is contained in $\mathcal U_{r}(p)$. 
Moreover, the estimate
\begin{equation}\label{u estimates v}
\| v\|_{\infty} \leq \| u\|_{\infty} + \sqrt 2 \ep r
\end{equation}
holds.
\end{lemma}

\begin{proof} 
Let $q=p+x+u(x)\nu_p $ be a point in $\mathcal U_{r}(p)$, with
\begin{equation}\label{q leq}
|x| < r \sqrt{1-\ep^2}.
\end{equation}
By the implicit function theorem, if  $\nu_q\cdot \ell >0$, then $S$ can be locally represented as a graph  of function near $q$ over the hyperplane $\ell^\perp$. Let $A \in {\rm SO}(n+1)$ be a special orthogonal matrix such that
\begin{equation*}
A \,\nu_p=\ell,
\end{equation*}
and let $y \in \ell^\perp$ be such that
\begin{equation*}
y=Ax.
\end{equation*}
Since $A \in {\rm SO}(n+1)$ we have
$|x|= |y|$
and then
$$
|y| < r \sqrt{1-\ep^2}\,.
$$
From triangle and Cauchy-Schwarz inequalities we have that
\begin{equation*}
\nu_q \cdot \ell \geq \nu_q\cdot \nu_p - |\ell -\nu_p|\,;
\end{equation*}
\eqref{bound on nu N+1} and \eqref{omega assumptions} yield that
\begin{equation*}
\nu_q \cdot \ell \geq \sqrt{1-\frac{|x|^2}{\rho^2}} - \ep,
\end{equation*}
which implies that $\nu_q \cdot \ell>0$ on account of \eqref{q leq}. Therefore any point $q \in V$ can be written both as $q=p+x+u(x)\nu_p$ and as $q=p+ y + v(y) \ell$ for some $x\in T_pS$ and $y\in \ell^{\perp}$. In particular
\begin{equation*}
y + v(y) \ell =x+u(x)\nu_p
\end{equation*}
and, since $y=Ax$, we have
\begin{equation*}
(I-A)x + u(x) \nu_p = v(y)\ell.
\end{equation*}
By taking the scalar product with $\ell$, we readily obtain
\begin{equation}\label{v leq I-A}
|v(\xi)| \leq |I-A| |x| + |u(x)|.
\end{equation}
The matrix $A$ can be choosen such that
\begin{equation*}
| I-A| \leq 2 \sqrt{1-\ell \cdot \nu_p} \leq \sqrt{2} \ep,
\end{equation*}
and \eqref{v leq I-A} implies the last part of the statement.
\end{proof}

It will be important to compare the normal vectors to two surfaces which are graphs of function over the same domain. We have the following Lemma.

\begin{lemma} \label{lemma diff normali}
Let $u_1,u_2\in C^{1}(B_r \cap e_{n+1}^\perp)$ and assume that
\begin{equation*}
|\nabla u_2(x_0) -  \nabla u_1(x_0)| < \ep,
\end{equation*}
for some $x_0 \in B_r \cap e_{n+1}^\perp$. Let $p_i = x_0 + u_i(x_0) e_{n+1}$, $i=1,2$. Then 
\begin{equation}\label{diff normali}
|\nu_{p_1} - \nu_{p_2}| \leq\frac{\sqrt{5}}{2} \ep,
\end{equation}
where
\begin{equation*}
\nu_{p_i}= \frac{-\nabla u_i(x_0) + e_{n+1}}{\sqrt{1+|\nabla u_i(x_0)|^2}},
\end{equation*}
is the inward normal to the graph of $u_i$ at $p_i$, $i=1,2$.
\end{lemma}

\begin{proof}
Since the eigenvalues of the Hessian of the function $x\mapsto \sqrt{1+|x|^2}$ are uniformly bounded by 1, then its gradient is Lipschitz continuous with constant 1 and we have that
\begin{equation} \label{eq lipsch gradient bound}
\Big{|} \frac{\nabla u_1(x)}{\sqrt{1+|\nabla u_1(x)|^2}} - \frac{\nabla u_2(x)}{\sqrt{1+|\nabla u_2(x)|^2}} \Big{|} \leq |\nabla u_1(x) - \nabla u_2(x)|.
\end{equation}
Moreover, we have that
\begin{equation}\label{eq lagrange bound}
\Big{|} \frac{1}{\sqrt{1+|\nabla u_1(x)|^2}} - \frac{1}{\sqrt{1+|\nabla u_2(x)|^2}} \Big{|} \leq \frac{1}{2} \big{|} |\nabla u_1(x)| - |\nabla u_2(x)| \big{|}.
\end{equation}
From triangle inequality and from \eqref{eq lipsch gradient bound} and \eqref{eq lagrange bound} we readily obtain \eqref{diff normali}.
\end{proof}


\section{Proof of Theorem \ref{main}} \label{section proof main}
The proof of Theorem \ref{main} relies upon a quantitative study of the method of moving planes and it consists of several steps, as we sketch in the following.
\begin{itemize}
\item[Step 1.] We fix a direction $\om$, apply the method of moving planes, and find a critical position which defines a critical hyperplane $\pi_m$, as described in Subsection \ref{subsection Alexandrov}. By using the smallness of $\oscH$, we can prove that (up to a connected component) the surface $S$ and the reflected cap $S_m$ are close. Hence, the union of the cap and the reflected cap provides a symmetric set in the direction $\om$ which gives information about the approximate symmetry of $S$ in the direction $\om$. It is important to notice that the estimates do not depend on the chosen direction.
\item[Step 2.] We apply Step 1 in $n+1$ orthogonal directions and we obtain a point $\mathcal{O}$ as the intersection  of the corresponding $n+1$ critical hyperplanes. Since the estimates in Step 1 do not depend on the direction, the point $\mathcal{O}$ can be chosen as an approximate center of symmetry. Moreover, any critical hyperplane in any other direction is far from $\mathcal{O}$ less than some constant times $\oscH$.
\item[Step 3.] Again by using the estimates in Step 1, we can define two balls centered at $\mathcal{O}$ such that estimate \eqref{stability radii} holds. 
\end{itemize}

We notice that once we have the approximate symmetry in one direction, i.e. Step 1, then the argument for proving Steps 2 and 3 is well-established (see \cite[Section 4]{ABR}). In the following we will prove Step 1, which is our main result of this section and, for the sake of completeness, we give a sketch of the proof for Steps 2 and 3.

%
%
%

\subsection{Step 1. Approximate symmetry in one direction}
We apply the moving plane procedure as described in Subsection \ref{subsection Alexandrov}.
Let $\omega\in \Sbb^n$ be a direction in $\RR^{n+1}$ and let
$S_m$,  $\hat S_m$ be defined as in \eqref{definitions}. Let 

\noindent \begin{center} $p_0$ be a tangency point between $S_m$ and $\hat S_m$, \end{center} 

\noindent and denote by $\Sigma$ and $\hat \Sigma$ the connected components of $S_m$ and $\hat S_m$, respectively,  containing $p_0$ or having $p_0$ on their boundary. Let $S^*$ be the reflection of $S$ about $\pi_m$.
For a point $p$ in $S$ (or $S^*$), we denote by $\nu_p$ the normal vector to $S$ (or to $S^*$) at $p$. We will use this notation when it does not create ambiguity: the choice of the vector normal and of the surface is implied by the point itself. When $p \in S \cap S^*$ is a point of tangency between $S$ and $S^*$, then the normal vector at $p$ is the same for both the surfaces, and the notation is coherent. When this notation creates an ambiguity, i.e. for nontangency points in $S \cap S^*$, we will specify the dependency on the surface. For points on $\pa \Sigma$ (or $\pa \hat \Sigma$) we will denote by $\nu$ the Gauss map on $\pa \Sigma$ (or $\pa \hat \Sigma$) which is induced by the one on $S^*$ (or $S$).


The main goal of Step 1 is to prove the following result of approximate symmetry in one direction.

\begin{theorem} \label{thm approx symmetry 1 direction}
There exists a positive constant $\ep$ such that if
$$
{\rm osc}( H) \leq \ep,
$$
then for any $p \in\Sigma$ there exists $\hat p\in \hat\Sigma$ such that
\begin{equation}\label{bound on dist}
|p-\hat p| + |\nu_p-\nu_{\hat p}| \leq C\, \oscH  . 
\end{equation}
Here, the constants $\ep$ and $C$ depend only on $n$, $\rho$, $|S|$ and do not depend on the direction $\om$.
\end{theorem}

Before giving the proof of Theorem \ref{thm approx symmetry 1 direction}, we provide two preliminary results
about the geometry of $\Sigma$. For $t> 0$ we set  
$$ 
\Sigma^t = \{p \in \Sigma:\ d_\Sigma(p,\pa \Sigma) > t \}\,.
$$
The following two lemmas show some conditions implying that $\Sigma^t$ is connected for $t$ small enough.

\begin{lemma} \label{lemma connected}
Assume that there exists $\mu\leq \frac12 $ such that 
\begin{equation}\label{trasversale}
\nu_p \cdot \om \leq \mu
\end{equation}
for every $p$ on the boundary of $\Sigma$.
Then $\Sigma^t$ is connected for any $0<t \leq t_0$, where
$$
t_0=\frac{\rho}{2\sqrt{n}}\sqrt{1-2\mu^2}.
$$
\end{lemma}
\begin{proof}
Let $S^*$ be the reflection of $S$ about $\pi_m$. We notice that, by construction of the method of moving planes, $\Sigma$ and $\pi_m$ enclose a bounded simply connected domain of $\RR^{n+1}$. Moreover, $\nu_p \cdot \om \geq 0$ on $\pa \Sigma$ and equation \eqref{trasversale} implies that $\pi_m$ intersects $S^*$ transversally. Hence, the boundary of $\Sigma$ is a manifold of class $C^2$. We prove that the boundary of $\Sigma^t$ lies in a tubular neighbourhood of the boundary of $ \Sigma$ in $S^*$.  Then, since $\Sigma$ is connected, every two points in $\Sigma^t$ can be joined by a curve in $\Sigma$ which can be pushed into $\Sigma^t$ by using the normal vector field to the boundary $\Sigma$.


According to Section \ref{subsect Luigi}, we denote the boundary of $\Sigma$ by $\Sigma'$ and we orient $\Sigma'$ by the Gauss map satisfying
$$
\nu_p \cdot \nu_p' = 1 - (\nu_p \cdot \om)^2
$$
(see formula \eqref{nu primo}). Hence, from \eqref{trasversale}, we have that
$$
\nu_p \cdot \nu_p' \geq 1 - \mu^2.
$$
Since the principal curvatures of $S$ are bounded by $\rho^{-1}$, from Proposition \ref{prop Luigi I} the principal  curvatures $\kappa_i'$ of $\Sigma'$ satisfy
\begin{equation} \label{ki prime 42}
|\kappa_i'| \leq \frac{1}{\rho(1- \mu^2)}, \quad i=1,\ldots,n-1.
\end{equation}
From Lemma \ref{lemma change normal}, we can write $S^*$ as a graph of function $u: B_{r} \cap (\nu_p')^\perp \to \RR$, with $r=\rho \sqrt{1-2\mu^2}$.  Moreover, \eqref{ki prime 42} and Lemma \ref{lemma bounds rho} yield that $\Sigma'$ is locally the graph of $u$ restricted to $B_{r} \cap T_p {\Sigma'}$. 
Taking into account that  $(\nu_p')^\perp= T_p\Sigma' \oplus \langle \om \rangle$, we consider the subset of $S^*$ given by
$$
Q(p)=\{q=p+ \xi +  s \om + u( \xi +  s \om) \nu_p' \, : \ \ \xi \in B_{r} \cap T_p \Sigma', \ |s| \leq t_0  \},
$$
which contains a tubular neighborhood of $\Sigma' \cap B_{t_0}(p)$ of radius at least $t_0$.
Hence, the set
$$
\mathcal{Q} = \bigcup_{p \in \Sigma'} Q(p)
$$
contains a tubular neighborhood of $\Sigma'$ in $S^*$ of radius at least $t_0$ and we conclude. 
\end{proof}

\begin{lemma} \label{lemma connected II}
Let $0< \de \leq \rho(8\sqrt{n})^{-1}$. If we suppose that there exists a connected component $\Gamma^\de$ of $\Sigma^\de$ satisfying 
$$
0 \leq \nu_p \cdot \om \leq \frac{1}{8},
$$  
then, $\Sigma^\de$ is connected.
\end{lemma}

\begin{proof}
In order to simplify the notation we let $\mu_0=1/8$. Notice that the interior and exterior touching balls at every boundary points of $\Gamma^\de$ intersects $\pi_m$. By using this argument and after elementary but tedious calculations, we can prove that any $q \in \Sigma \setminus {\Gamma}^\de$ is such that 
$$d_\Sigma(q, \Gamma^\de) \leq \rho \arcsin \left((1+ 2\mu_0) \frac{\de}{\rho}\right).$$
In particular, for any $q \in \pa \Sigma$ there exists $p\in \pa \Sigma^\de$ such that 
$$d_\Sigma(q, p) \leq \rho \arcsin \left((1+ 2\mu_0) \frac{\de}{\rho}\right),$$
and from Lemma \ref{lemma bounds rho} we obtain that 
$$
|\nu_p-\nu_q| \leq \sqrt{2} \arcsin \left((1+ 2\mu_0) \frac{\de}{\rho}\right).
$$
By writing $\nu_q \cdot \om = \nu_p \cdot \om - (\nu_q-\nu_p) \cdot \om$ and by triangle inequality we get
$$
|\nu_q \cdot \om| \leq \mu_0 +  \sqrt{2} \arcsin \left((1+ 2\mu_0) \frac{\de}{\rho}\right);
$$
our assumptions on $\de$ implies the following (rougher but simpler) bound: 
$$
|\nu_q \cdot \om| \leq 2\mu_0 +  \frac{1}{2}.
$$
Now we use Lemma \ref{lemma connected} by setting $\mu=2 \mu_0+1/2$ and imposing that $\de \leq t_0$, and we conclude.
\end{proof}

Now, we focus on the proof of Theorem \ref{thm approx symmetry 1 direction}. It will be divided in four cases, which we study in the following subsections. In each case, $\de$ will be fixed to be  
$$ \de= \min \left(\frac{\rho}{2^6}, \frac{\rho}{8\sqrt{n}}\right).$$
Moreover, the constants $\ep$ and $C$ can be chosen as 
$$
\ep= \min \{\ep_0,\ep_1,\ep_2, \ep_3 \},
$$
and 
$$
C=\frac{5}{4} C_1 K_1 K_2 K_3,
$$
respectively. Here, $\ep_0$ is given by \eqref{epsilon 0}, and $\ep_1, \ep_2, \ep_3$ and $C_1$ will be defined in the following. Moreover, $K_1, K_2, K_3$ are given by Lemmas \ref{lemma Harnack interior}, \ref{lemma Carleson}, \ref{lemma hopf}, respectively, where $M$ is chosen accordingly to Lemma \ref{lemma bounds rho} by assuming that $|x| \leq \rho/2$. Hence, the constants $\ep$ and $C$ depend only on $n$ and upper bounds on $\rho^{-1}$ and $|S|$.

\subsubsection{Case 1. $d_{\Sigma}(p_0,\pa \Sigma) > \de$ and $d_{\Sigma}(p,\pa \Sigma) \geq \de$} \label{subsec case1}

In this first case we assume that $p_0$ and $p$ are interior points of $\Sigma$, which are far from $\pa \Sigma$ more than $\de$. 
We remark that in this case, $p_0$ is an interior touching point between $\Sigma$ and $\hat\Sigma$, so that case (i) in the method of moving planes occurs. We first assume that $p_0$ and $p$ are in the same connected component of $\Sigma^\de$; then, Lemma \ref{lemma connected II} will be used in order to show that $\Sigma^\de$ is in fact connected. 

Let
\begin{equation*}
    r_0=\rho \sin \frac{\de}{2\rho}.
\end{equation*}
Since $p$ and $p_0$ are in a connected component of $\Sigma^\de$, there exist: $\{p_1,\dots,p_N\}$ in the connected component of $\Sigma^{\delta/2}$ containing $p_0$, a chain $\{\mathcal U_{r_0}(p_i)\}_{\{i=0,\ldots,N\}}$ of open sets of $\Sigma$ and a sequence of maps $u_i: \ B_{r_0}\cap T_{ p_i}\Sigma \to \RR,\ i=0,\ldots,N$,
as in Lemma \ref{lemma harnack chain}, where $r_i=(1-\ep)^i r_0$. We notice that $\Sigma$ and  $\hat\Sigma$ are tangent at
$p_0$ and that in particular the two normal vectors to $\Sigma$ and $\hat \Sigma$ at $p_0$ coincide. We stress that $\hat \Sigma \subset S$ and that, since $r_0<\rho$, from Lemma \ref{lemma bounds rho} we have that $S$  is locally represented near $p_0$  as a graph of a map
$\hat u_0\colon B_{r_0}\cap T_{p_0}S\to \RR$. 

%
Lemma \ref{lemma bounds rho} implies that $| \nabla u_0|, |\nabla \hat u_0 | \leq M$ in $B_{r_0} \cap T_{p_0} \Sigma$, where $M$ is some constant which depends only on $r_0$, i.e. only on $\rho$. Now, we use Lemma \ref{lemma Harnack interior}: since $u_0(0)=\hat u_0(0)$
and $u_0 \geq \hat u_0$, \eqref{eq Harnack C1} gives
\begin{equation}\label{harnack step 1}
\|u_0- \hat u_0\|_{C^1(B_{r_0/4}\cap T_{p_0} \Sigma)} \leq K_1 \,{\rm osc}( H),
\end{equation}
where $K_1$ depends only on $n$ and $M$. We notice that from \eqref{Vr0 subset} we have that $p_1\in\overline{ \mathcal U}_{r_0/4}(p_0)$.  Let $x_1$ be the projection of $p_1$ onto $T_{p_0}\Sigma$ and let
$$
\hat p_1:=p_0+x_1+\hat u_{0}(x_1)\nu_{p_0}\in \hat \Sigma\,.
$$
From \eqref{harnack step 1} we obtain that
\begin{equation*}
|\nabla u_0(x_1)-\nabla \hat u_0(x_1)| \leq K_1\,{\rm osc}(H),
\end{equation*}
and therefore
Lemma \ref{lemma diff normali} yields
\begin{equation}\label{norm diff step 1}
|\nu_{p_1} - \nu_{\hat p_1} | \leq  \frac{\sqrt{5}}{2}K_1\, \oscH .
\end{equation}

As already mentioned, we have a local parametrization of $\Sigma$ in a neighborhood of $p_1$ as a graph of the $C^2$ function $u_1\colon B_{r_0}\cap T_{p_1}\Sigma\to \RR$. Lemma \ref{lemma change normal} and \eqref{norm diff step 1} imply that $S$ can be locally parameterized by a graph of function $\hat u_1\colon B_{r_1}\cap T_{p_1} \Sigma \to \RR$, being $r_1 < r_0 \sqrt{1-\frac{5}{4} K_1^2 \ep^2}$, since $\ep \leq \ep_1$ with 
\begin{equation} \label{ep 1}
\ep_1 = \left( 1 + \frac 5 4 K_1^2 \right)^{-1}.
\end{equation}
Moreover, \eqref{u estimates v} yields that
\begin{equation*}
|u_1(0)-\hat u_1(0) | \leq \|u_0-\hat u_0\|_{C^0(B_{r_0/4}\cap T_{p_0}\Sigma)} + \sqrt{5} r_0 K_1 \oscH;
\end{equation*}
from \eqref{harnack step 1} and since $u_1-\hat u_1 \geq 0$ by construction, we find that
\begin{equation*}
0 \leq u_1(0) - \hat u_1(0) \leq (1+r_0\sqrt{5}) K_1 \oscH.
\end{equation*}
We use Lemma \ref{lemma Harnack interior} and obtain that
\begin{equation}\label{harnack step 2}
\|u_1-\hat u_1\|_{C^1(B_{r_1/4}\cap T_{p_1} \Sigma)} \leq K_1[(1+r_0\sqrt{5}) K_1 +1] \,\oscH.
\end{equation}
Now, \eqref{harnack step 2} is the analogue of \eqref{harnack step 1} with $p_1$ instead of $p_0$, and we can iterate until we obtain two functions
\begin{equation*}
u_N,\hat u_N:\ B_{r_N}\cap T_p \Sigma \to \RR,
\end{equation*}
such that
\begin{equation} \label{harnack step final}
\| u_N- \hat u_N\|_{C^1(B_{r_N /4}\cap T_p \Sigma)} \leq C_1 \,{\rm osc}(H).
\end{equation}
A choice of $\hat p$ as in the statement of Theorem \ref{thm approx symmetry 1 direction} is then given by  
$$
\hat p=p + \hat u_N(0) \nu_p\,,
$$ 
since \eqref{bound on dist} is implied by \eqref{harnack step final} and Lemma \ref{lemma diff normali}.

We notice that a choice of the constant $C_1$ in \eqref{harnack step final} is given by
\begin{equation}\label{C1}
C_1=\Big( (1+r_0\sqrt{5})K_1 + 1 \Big)^{N_0+1},
\end{equation}
where $N_0$ is given by \eqref{n 0}. Hence the constant $C_1$ depends only on $n$, $\de/\rho$, and an upper bound on $|S|$.

Once we have \eqref{harnack step final} for any $p$ in a connected component of $\Sigma^\de$, we have in fact that 
$$
\nu_q \cdot \om \leq \frac 1 8,
$$
for any point $q$ at the boundary of such a connected component, as it is implied by the following Lemma.

\begin{lemma} \label{lemma bound on nu1}
Let $q\in \Sigma $ be such that $d_\Sigma(q,\pa \Sigma) \leq \de$. Assume that the point
\begin{equation*}
\hat q=q-\alpha \nu_q
\end{equation*}
is on $\hat \Sigma$ and is such that
\begin{equation}\label{nu p - nu q}
|\nu_q-\nu_{\hat q}| \leq \alpha,
\end{equation}
with $\alpha + 2\de < \rho$. Then we have that
\begin{equation} \label{nu p cdot om leq}
0 \leq \nu_q \cdot \om \leq \sqrt{8 \frac{\de^2}{\rho^2} + \frac{\alpha}{2}}.
\end{equation}
\end{lemma}

\begin{proof}
Let $q^m$ be the reflection of $q$ about $\pi_m$ and let
$$
t=\nu_q \cdot \om.
$$
By construction of the method of moving planes, it is clear that $t \geq 0$ and the first inequality in \eqref{nu p cdot om leq} follows.
We denote by $\nu_{q^m}$ the inner normal vector to $S$ at $q^m$. 
Since $\nu_q \cdot \om = - \nu_{q^m} \cdot \om$ and $\nu_q-\nu_{q^m}= 2 t \om$, we have that
\begin{equation}\label{nu p nu bar p}
\nu_q \cdot \nu_{q^m} = 1 -2t^2.
\end{equation}
We notice that $q^m$ and $\hat q$ both lie in  $S$ and  $|q^m- \hat q| \leq \alpha + 2 \de$, which implies that $\hat q \in \mathcal{U}_{\rho}(q^m)$ provided that $\alpha + 2 \de < \rho$. Hence, \eqref{bound on nu N+1} yields 
that
\begin{equation*}
\nu_{\hat q}\cdot \nu_{q^m} \geq \sqrt{1- \left(\frac{\alpha +2\de}{\rho}\right)^2}.
\end{equation*}
From \eqref{nu p - nu q} and \eqref{nu p nu bar p} we find that
\begin{equation*}
1-2t^2 \geq \sqrt{1- \left(\frac{\alpha +2\de}{\rho}\right)^2} - \alpha,
\end{equation*}
which gives
$$t^2 \leq \frac{1}{2} \Big( \frac{\alpha + 2\de}{\rho} \Big)^2 + \frac{\alpha}{2},$$
and we obtain the second inequality in \eqref{nu p cdot om leq}.
\end{proof}

The conclusion of Case 1 follows from the following argument. From  \eqref{harnack step final} we know that for any $q$ on the boundary of the connected component of $\Sigma^\de$ containing $p_0$ there exists $\hat{q} \in \hat{\Sigma}$ such that 
$$
|q-\hat q| + |\nu_q - \nu_{\hat q}| \leq C_1 \oscH. 
$$
We apply Lemma \ref{lemma bound on nu1} by letting $\al=C_1 \oscH$ and, since $\ep \leq \ep_2$, with
$$
\ep_2 \leq \frac{1}{2^6 C_1},
$$ 
we obtain that $0 \leq \nu_q \cdot \om \leq 1/8$. Hence, from Lemma \ref{lemma connected II} we have that $\Sigma^\de$ is connected and we conclude.

\medskip
\subsubsection{Case 2: $d_\Sigma(p_0,\pa \Sigma) \geq \de$ and $d_\Sigma(p,\pa \Sigma) < \de$} \label{subsec case2}
Here the idea consists in extending the estimate in Subsection \ref{subsec case1} to the whole $\Sigma$. This will be done by using Carleson type estimates given by Lemma \ref{lemma Carleson}. We remark that its application is not trivial, since we need more information on how $S$ intersects $\pi_m$.

\medskip 
Accordingly to \eqref{definitions} in Subsection \ref{subsection Alexandrov}, for a given point $p \in\Sigma$ such that $d_\Sigma (p, \pa \Sigma) \leq \de$, we denote  by  $p^m$ the point of $S$ obtained by reflecting $p$ about  $\pi_m$. $S$ can be locally written as a graph of function $u\colon  B_{\rho}\cap T_p S \to\RR$. For $0<r<\rho$,  we define ${U}_r^*(p)$ as the reflection of $\mathcal{U}_r(p^m)$ about $\pi_m$ and we denote by $U_r(p)$ the subset of $\Sigma$ obtained by
$$U_r(p) = {U}_r^*(p) \cap  \{q\in \RR^{n+1}\,\,:\,\, q\cdot \omega<m\}.$$ 
Moreover, we denote by $E_r$ the open subset of  $B_r\cap T_p\Sigma$ such that 
\begin{equation} \label{U al bordo}
U_r(p) = \{ p + x + u(x) \nu_p:\ x \in E_r\}\,.
\end{equation}


The next result is a consequence of Propositions \ref{prop Luigi I}, \ref{prop Luigi II} in Subsection \ref{subsect Luigi}. 

\begin{lemma} \label{coroll curvatures}
Let $q\in \Sigma $ be such that $d_\Sigma(q,\pa \Sigma) = \de$ and $ 0 \leq \nu_q \cdot \om \leq  1/4$. Let
$U'={U}^*_{\sqrt{2} \rho /8}(q)\cap \pi_m$ and $U''$ be the orthogonal projection of $U'$ onto $T_q\Sigma$.
Then $U''$ is a hypersurface of class $C^2$ of $T_q\Sigma$ whose principal curvatures are bounded by 
$$ \mathcal{K} =  \frac{4\de}{\rho^2}.  $$
\end{lemma}

\begin{proof} We notice that since $d_\Sigma(q,\pa \Sigma) = \de$ then $U' \neq \emptyset$.  Let $\zeta \in U'$ be arbitrary.  
Since the projection ${\rm pr}(\zeta)$ of $\zeta$ on $T_q\Sigma$ is in $\overline{B}_{\sqrt{2} \rho /8}$, from \eqref{bound on nu N+1} in Lemma \ref{lemma bounds rho} we know that
\begin{equation} \label{nup - nuq coroll}
 |\nu_q-\nu_\zeta| \leq \frac{1}{4} .
\end{equation}
Since $\nu_\zeta \cdot \om = \nu_q \cdot \om + (\nu_\zeta -\nu_q) \cdot \om$, we have that
\begin{equation} \label{nuq cdot om cor}
|\nu_\zeta \cdot \om | \leq \frac{1}{2} ,
\end{equation}
which implies that $\pi_m$ intersects $U^*_{\sqrt{2} \rho /8}(q)$ transversally,
and so $U''$ is a hypersurface of $T_q\Sigma$. 
Since the principal curvatures of $S$ are bounded by $1/\rho$, \eqref{bound curvatures II} implies that the principal curvatures of $U''$ satisfy
\begin{equation*}
|\kappa''_i({\rm pr}(\zeta))|  \leq \frac{1}{\rho |\nu_\zeta \cdot \nu_\zeta'|} \cdot \frac{ \om\cdot \nu_q }{ [(\omega\cdot\nu_q)^2+(\nu_q\cdot \nu'_\zeta)^2]^{3/2}}\,,\quad i=1,\dots,n-1,
\end{equation*}
where $\nu'$ is the Gauss map of $U'$ viewed as a hypersurface of $\pi_m$ satisfying 
\begin{equation}
\label{eq lemma curv A}
\nu_\zeta\cdot \nu'_\zeta  = 1-  (\nu_\zeta\cdot \om )^2\, .
\end{equation} 
Hence,
\begin{equation} \label{k'' leq}
|\kappa''_i({\rm pr}(\zeta))|  \leq \frac{ \om\cdot \nu_q }{\rho |\nu_\zeta \cdot \nu_\zeta' | \,  |\nu_q\cdot \nu'_\zeta|^{3}}\,,\quad i=1,\dots,n-1\,. 
\end{equation}
From \eqref{nuq cdot om cor} and \eqref{eq lemma curv A}, we obtain that 
\begin{equation} \label{nuova}
\nu_\zeta\cdot \nu'_\zeta \geq \frac 3 4  \, .
\end{equation}
By writing 
\begin{equation*}
\nu_q \cdot \nu'_\zeta= (\nu_q-\nu_\zeta) \cdot \nu'_\zeta + \nu_\zeta \cdot \nu'_\zeta,
\end{equation*}
and using \eqref{nup - nuq coroll} and \eqref{nuova}
we get
\begin{equation*}
\nu_q \cdot \nu'_\zeta \geq  \frac{1}{2} ,
\end{equation*}
and from \eqref{k'' leq} and \eqref{nuova} we conclude. 
\end{proof}


In the next lemma
we give a bound which will be useful in the sequel.

\begin{lemma} \label{lemma bound on nu1 II}
Let $q$ and $\al$ be as in Lemma \ref{lemma bound on nu1}. Then, we have that
\begin{equation} \label{nu ell cdot om leq}
0 \leq \nu_\zeta \cdot \om \leq \sqrt{8 \frac{\de^2}{\rho^2} + \frac{\alpha}{2}} + \frac{\sqrt{2}}{\rho} d_\Sigma(q,\zeta),
\end{equation}
for any $\zeta \in\overline{ U}_{\rho}(q)$, where $U_\rho(q)$ is defined as in \eqref{U al bordo}.
\end{lemma}

\begin{proof}
%
Let $\zeta \in \overline{U}_\rho(q)$. By construction we have that $\nu_\zeta \cdot \om \geq 0$. Since 
\begin{equation*}
\nu_\zeta \cdot  \om \leq \nu_q \cdot \om + |\nu_\zeta - \nu_q|,
\end{equation*}
from  \eqref{bound on nu N+1} and \eqref{nu p cdot om leq} we conclude.
\end{proof}

\noindent Now we are ready to prove Theorem \ref{thm approx symmetry 1 direction} for Case 2.
Let 
$$
\ep_3= \frac{\de}{\rho C_1} 
$$
where $C_1$ is given by \eqref{C1}.
We assume that $d_\Sigma (p_0, \pa \Sigma) \geq \de$ and $d_\Sigma (p, \pa \Sigma) < \de$. By arguing as in Case 1, we have that $\Sigma^\de$ is connected. Let $q \in \Sigma$ and $\bar p \in \pa \Sigma$ be such that
$$d_\Sigma (p, q) +  d_\Sigma (p,\pa \Sigma)  = \de ,$$
and
$$d_\Sigma (p, \bar p) = d_\Sigma (p,\pa \Sigma)$$
(we notice that our choice of $\de$ implies that $q$ and $\bar p$ exist).

Since $ d_\Sigma (q,\pa \Sigma) = \de$, from Case 1 we have that there exists $\hat q \in \hat \Sigma$ such that
\begin{equation} \label{case 2 diff q}
|q-\hat q| + |\nu_q-\nu_{\hat q}|\leq C_1 \oscH
\end{equation}
(see formula \eqref{harnack step final}).
From the proof of Case 1, it is clear that $\hat q$ can be chosen as
\begin{equation*}
\hat q= q-\al \nu_q,
\end{equation*}
for some $0 \leq \al \leq C_1 \oscH$. Let
\begin{equation} \label{r case 2}
r= \frac \rho 8.
\end{equation}
We define the sets $U_{r}(q)\subset \Sigma$, $E_{r}\subseteq B_{r}\cap  T_{q}\Sigma$,  and the map $u:E_r \to \RR$ as in \eqref{U al bordo} with $q$ in place of $p$. Since $\hat q \in \hat \Sigma \subset S$ and $|\nu_q-\nu_{\hat q}|\leq C_1 \oscH$, from Lemma \ref{lemma change normal} we have that $S$ can be locally written (around $\hat q$) as a graph of function $\hat u$ over $T_q \Sigma \cap B_{\rho \sqrt{1-C_1^2 \ep_3^2}}$ and in particular over $T_q \Sigma \cap B_{r}$ (which is justified by our choice of $\ep_3$). 

We notice that Lemma \ref{lemma bound on d Ga} implies that $p,\bar p \in \overline{U}_r(q)$.
Let $ \pa E_r$ be the boundary of $E_r$ in $T_q \Sigma$  and let $\bar x \in \pa E_r$ be the projection of $\bar p$. Since $d_\Sigma(q , \bar p) = \de$, from Lemma \ref{lemma bound on d Ga} we have that
\begin{equation} \label{x0 geq leq}
\rho\sin \frac{\de}{\rho} \leq |\bar x| \leq \de.
\end{equation}
Let $U'=U^*_r(q) \cap \pi_m$ and let $U''$ be  the projection of $U'$ onto $T_{q}\Sigma$ (as in Lemma \ref{coroll curvatures}). Notice that by definition $U''$  is contained in $\partial E_r$ and, in particular, $u=\hat u$ on $U''$. 
From Lemma \ref{lemma bound on nu1} and Lemma \ref{coroll curvatures}, we have that the principal curvatures of $U''$ are uniformly bounded by $\mathcal{K}$.
We notice that our choice of $\de$ implies that $\mathcal{K} \leq \frac{1}{16 \rho} $.

\begin{figure}
\begin{tikzpicture} [scale=1.5]
\fill[gray!20!white] ({0.28+0.4*cos(88)}, {0.15+0.4*sin(88)})-- ({0.28+0.4*cos(260)}, {0.15+0.4*sin(260)}) ..  controls({0.28+0.4*cos(180)}, {0.15+0.4*sin(180)},0.15)..({0.28+0.4*cos(88)}, {0.15+0.4*sin(88)}) ;
\fill[gray!20!white] ({0.28+0.4*cos(260)}, {0.15+0.4*sin(260)})-- ({0.28+0.4*cos(88)}, {0.15+0.4*sin(88)}) arc  (88:260:4mm);
\node at (-2,1.8) { $T_q\Sigma$};
\node at (-1.5,0.6) { $E_r$};
\node at (0.2,1.6) { $U''$};
\node at (-0.43,1.3) {\footnotesize $4\delta$};
\node at (-0.38,0.4) {\footnotesize $2\delta$};
\filldraw [black] (0,0) circle (0.5 pt) node[left] {\footnotesize $O$};
\filldraw [black] (0.28,0.15) circle (0.5 pt) node[right] {\footnotesize $\bar x$};
\filldraw [black] ({0.28+0.8*cos(171)}, {0.15+0.8*sin(171)}) circle (0.5 pt) node[left] {\footnotesize $\bar y$};
\node at (0.16,-0.1) {\footnotesize $x$};
\filldraw [black] (0.15,0.05) circle (0.5 pt);
\draw [domain=100:260] plot ({2*cos(\x)}, {2*sin(\x)});
\draw [domain=88:260] plot ({0.28+0.4*cos(\x)}, {0.15+0.4*sin(\x)});
\draw [domain=107:255] plot ({0.28+1.6*cos(\x)}, {0.15+1.6*sin(\x)});
\draw (0.28,0.15) -- ({0.28+1.5*cos(115)}, {0.15+1.5*sin(105)});
\draw (0.28,0.15) --({0.28+0.8*cos(171)}, {0.15+0.8*sin(171)});
\draw ({2*cos(100)}, {2*sin(100)}) ..  controls(1/2,1/2).. ({2*cos(260)}, {2*sin(260)});
\draw({0.28+0.8*cos(171)}, {0.15+0.8*sin(171)}) circle (0.8);
\end{tikzpicture}
\caption{Case 2 in the proof of Theorem \ref{thm approx symmetry 1 direction}. The shadow region is the set $B_{\de}(\bar x) \cap E_r$.}
\label{fig case 2}
\end{figure}


Let $x$ be the projection of $p$ over $T_q\Sigma $. From \eqref{x0 geq leq} we have that $B_{4\de}(\bar x) \cap \pa E_r \subset U''$ and we can apply Lemma \ref{lemma Carleson} and obtain that
\begin{equation} \label{vivaLuigi}
\sup_{B_{\de}(\bar x) \cap E_r} (u-\hat u) \leq K_2 \left((u-\hat u)(\bar y)+ \oscH \right),
\end{equation}
with $\bar y = \bar x + 2\de \nu''_{\bar x}$, where $\nu''_{\bar x}$ is the interior normal to $U''$ at $\bar x$  (see Figure \ref{fig case 2}). We notice that $x \in B_{\de}(\bar x) \cap E_r$ and then from \eqref{vivaLuigi} we have that
\begin{equation} \label{bound case2 final}
(u-\hat u)(x) \leq K_2 \left((u-\hat u)(\bar y)+ \oscH \right).
\end{equation}
Since $2\de < \mathcal{K}^{-1}$, the point $\bar y$ has distance $2\de$ from the boundary of $E_r$ and, from Lemma \ref{lemma bound on d Ga} we have that the point 
$$
\bar q = q + \bar y + u(\bar y) \nu_q
$$ 
is such that 
$$
d_\Sigma(\bar q,\pa \Sigma) \geq 2\de \, .
$$ 
Hence, from Case 1 (applied to $p_0$ and $\bar q$) we obtain the estimate
$$
(u-\hat u)(\bar y) \leq  C_1 \oscH,
$$
and from \eqref{bound case2 final} we get 
\begin{equation*} 
(u-\hat u)(x) \leq C_1  K_2 \oscH.
\end{equation*}
By letting $\hat p = q + x + \hat u(x) \nu_q$, and since $d_\Sigma(p,\pa \Sigma) >0$, a standard application of Lemma \ref{lemma Harnack interior} and Lemma \ref{lemma diff normali} yield the estimate
$$
|p-\hat p| + |\nu_p - \nu_{\hat p}| \leq \frac{\sqrt{5}}{2} C_1 K_1 K_2 \oscH,
$$
and we complete the proof of Case 2.

\subsubsection{Case 3: $0 < d_\Sigma(p_0,\pa \Sigma) < \de$.}
Since $p_0$ is the tangency point, it is easy to show that the center of the interior touching sphere of radius $\rho$ to $S$ at $p_0$ lies in the half-space $\{q\in \RR^{n+1}\,:\,q\cdot \om \leq m\}$ (see for instance \cite[Lemma 2.1]{CMV}). From this, and being
$$|p_0 \cdot \om - m | \leq d_{\Sigma}(p_0, \pa \Sigma) \leq \de,$$
from Lemma \ref{lemma bound on nu1} (with $\al=0$) we have that
\begin{equation*}
\nu_{p_0} \cdot \om \leq 3\frac  \de\rho.
\end{equation*}

As for Case 2 (with $q$ replaced by $p_0$), we locally write $\Sigma$ and $\hat \Sigma$ as graphs of function $u, \hat u\colon E_r\to \RR$, respectively, where $E_r \subseteq T_{p_0} \Sigma$ is defined as in the introduction to this subsection, and $r$ is given by \eqref{r case 2}. Moreover, we denote by $U''$ the portion of $\pa E_r$ which is obtained by projecting ${U}^*_r(p_0) \cap \pi_m$ onto $T_{p_0} \Sigma$. We remark that $u=\hat u$ on $U''$ and that the principal curvatures of $U''$ are bounded by $\mathcal{K}$.

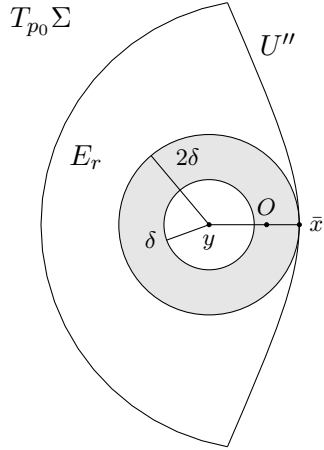
\begin{figure}
\begin{tikzpicture} [scale=1.5]
\fill[gray!20!white] ({-0.51}, {0}) circle (0.8);
\fill[white]({-0.51}, {0}) circle (0.4);
\node at (-2,1.8) { $T_{p_0}\Sigma$};
\node at (-1.6,0.6) { $E_r$};
\node at (0.1,1.6) {$U''$};
\node at (-0.7,0.6) {\footnotesize $2\delta$};
\filldraw [black] (0,0) circle (0.5 pt) node[above] {\footnotesize $O$};
\filldraw [black] (0.29,0) circle (0.5 pt) node[right] {\footnotesize $\bar x$};
\filldraw [black] ({0.29-0.8}, {0}) circle (0.5 pt) node[below] {\footnotesize $ y$};
\draw [domain=100:260] plot ({2*cos(\x)}, {2*sin(\x)});
\draw (0.29,0)  --({0.29-0.8}, {0});
\draw ({2*cos(100)}, {2*sin(100)}) ..  controls(1/2,0).. ({2*cos(260)}, {2*sin(260)});
\draw({-0.51}, {0}) circle (0.8);
\draw({-0.51}, {0}) -- ({-0.51+0.8*cos(130)},{0.8*sin(130)});
\draw({-0.51}, {0}) circle (0.4);
\draw({-0.51}, {0}) -- ({-0.51+0.4*cos(200)},{0.4*sin(200)})node[left] {\footnotesize $\delta$};;
\end{tikzpicture}
\caption{Case 3 in the proof of Theorem \ref{thm approx symmetry 1 direction}.}
\label{fig case 3}
\end{figure}

Let $\bar x \in U''$ be a point such that 
$$|\bar x| = \min_{x \in U''} |x|.$$
Notice that $|\bar x| \leq d_\Sigma (p_0, \pa \Sigma) < \de$. Let $\nu''_{\bar x}$ be the interior normal to $U''$ at $\bar x$,  and set
$$
y=\bar x +2\de \nu''_{\bar x}
$$
(see Figure \ref{fig case 3}). We notice that the principal curvatures of $U''$ are bounded by $\mathcal{K}$ and $2\de \leq \mathcal{K}^{-1}$ and the ball $B_{2\de}(y)\cap T_{p_0}\Sigma$ is contained in $E_r$ and tangent to $U''$ at $\bar x$, with $\nu_{\bar x}''= - \bar x / |\bar x|$. Hence, the origin $O$ of $T_{p_0} \Sigma$ (i.e. the projection of $p_0$ over $T_{p_0} \Sigma$) lies in the annulus $ (B_{2\de}(y) \setminus B_{\de}(y))\cap T_{p_0}\Sigma$. Hence, we can apply \eqref{Hopf A} in Lemma \ref{lemma hopf} (there we set: $x_0=\bar x$, $c=y$ and $r=2\de$) and, since $u(0)=\hat u (0)$, we find that
\begin{equation} \label{u - hat u case 3}
\|u-	\hat u\|_{C^1(B_{\de /2}(y) \cap T_{p_0} \Sigma)} \leq K_3 \oscH.
\end{equation}
Let 
$$
q=p_0+y+u(y) \nu_{p_0},\quad \textmd{and }\ \hat q=p_0+y+\hat u (y) \nu_{p_0}.
$$ 
We notice that from \eqref{u - hat u case 3} and Lemma \ref{lemma diff normali} we have that 
$$ |q-\hat q| + |\nu_q - \nu_{\hat q}| \leq \frac{\sqrt 5}{2} K_3 \oscH.$$
Since $y$ has distance $2\de$ from $\pa E_{r}$, then $d_\Sigma (q,\pa \Sigma) \geq 2 \de$, and we can apply Cases 1 and 2 to conclude.

\subsubsection{Case 4: $p_0 \in \pa \Sigma$.}
This case is the limiting case of Case 3 for $d_\Sigma(p_0,\pa \Sigma) \to 0$. Indeed, in this case we can write $\Sigma$ and $\hat \Sigma$ as graphs of functions over a half-ball on $T_{p_0} \Sigma$. Hence the argument used in Case $3$ can be adapted easily by using \eqref{Hopf B} instead of \eqref{Hopf A}. 

\subsection{Steps 2-3. Approximate radial symmetry and conclusion}
We consider $n+1$ orthogonal directions $e_1,\ldots,e_{n+1}$, and we denote by $\pi_1,\ldots,\pi_{n+1}$ the corresponding critical hyperplanes. Let 
$$
\mathcal{O} = \bigcap_{i=1}^{n+1} \pi_i,
$$
and denote by $\mathcal{R}(p)$ the reflection of $p$ in $\mathcal{O}$. We have the following Lemma which extends Theorem \ref{thm approx symmetry 1 direction}.

\begin{lemma} \label{lemma approx radial symmetry}
For any point $p \in S$ there exists a point $q\in S$ such that 
$$|\mathcal{R}(p) - q|  \leq (n+1)C \oscH. $$
\end{lemma}

\begin{proof}
We write $\mathcal{R}$ as 
$$
\mathcal{R}= \mathcal{R}_{n+1} \circ \cdots \circ \mathcal{R}_{1}, 
$$
where $\mathcal{R}_i$ is the reflection about $\pi_i$, $i=1,\ldots,N+1$. By iterating Theorem \ref{thm approx symmetry 1 direction} $n+1$ times, we conclude. 
\end{proof}

As in \cite[Proposition 6]{ABR} we have that, for every direction $\om$, it holds that 
\begin{equation} \label{estimate conclusion}
\dist(\mathcal{O},\pi_m) \leq C \oscH,
\end{equation}
where $\pi_m$ is the critical hyperplane in the direction $\om$ and $C$ is a constant that depends only on $\rho$ and $\diam S$, where
$$
\diam S = \max_{p,q \in S} |p-q|.
$$
We notice that $\diam S$ can be bounded in terms of $|S|$ and $\rho^{-1}$. Indeed, let $p,q \in S$ be such that $|p-q|= \diam S$. By arguing as in the proof of Lemma \ref{lemma piecewise geodesic}, we can find a piecewise geodesic path on $S$ joining $p$ and $q$, and with length bounded by \eqref{L number def} (with $\de=\rho/2$ there), and then
\begin{equation*}
\diam S \leq \frac{|S| 2^{2n}}{\om_n \rho^n}.
\end{equation*}
Hence, the constant $C$ in \eqref{estimate conclusion} can be bounded in terms of the dimension $n$ and upper bounds on $\rho^{-1}$ and $|S|$.

%

Finally, the bound on the difference of the radii \eqref{stability radii} of the approximating balls is obtained by arguing as in \cite[Proposition 7]{ABR}. Indeed, we define 
$$
r_i= \min_{p \in S} |p-\mathcal{O}|, \quad  \textmd{and }\ \ r_e= \max_{p \in S} |p-\mathcal{O}|,
$$
assume that the minimum and maximum are attained at $p_i$ and $p_e$, respectively, we obtain that 
$$
r_e-r_i \leq 2 \dist(\mathcal{O},\pi),
$$
where $\pi$ is the critical hyperplane in the direction 
$$\frac{p_e-p_i}{|p_e-p_i|}.$$
From \eqref{estimate conclusion} we conclude.

\section{Proof of Corollary \ref{main2}} \label{section proof main2}
\begin{lemma}\label{pre}
Let $S$ be a closed $C^2$ hypersurface embedded in $\RR^{n+1}$ and assume 
$$ 
S \subset \overline{B}_{r_e} \setminus B_{r_i},
$$
with $r_e-r_i \leq 2\rho$.
%
Then
$$
\frac{p}{|p|}\cdot \nu_p\leq -1+\frac{1}{\rho}(r_e-r_i)
$$
for every $p\in S$.
\end{lemma}
\begin{proof}
Let $p\in S$ and let $c^-$ and $c^+$ be the centers of the interior and the exterior touching balls of radius $\rho$ tangent at $p$, respectively. Then
$$
\left|c^-+\frac{c^-}{|c^-|}\rho\right|=\sup_{q\in B_\rho(c^-)}|q|\leq r_e \,,\quad \left|c^+-\frac{c^+}{|c^+|}\rho\right|=\inf_{q\in B_\rho(c^+)}|q|\geq r_i,
$$
and so
$$
\left|c^-+\frac{c^-}{|c^-|}\rho\right|^2-\left|c^+-\frac{c^+}{|c^+|}\rho\right|^2\leq r_e^2-r_i^2\,.
$$
Therefore
$$
|c^-|^2+2\rho|c^-|-|c^+|^2+2\rho |c^+|\leq r_e^2-r_i^2\,.
$$
Taking into account that
$$
c^+=p-\rho \nu_p\,,\quad c^-=p+\rho \nu_p,
$$
we get
$$
4\rho\,p\cdot \nu_p+2\rho(|c^-|+|c^+|)\leq  r_e^2-r_i^2,
$$
and so
$$
\frac{p}{|p|}\cdot \nu(p)\leq - \frac{|c^-|+|c^+|}{2|p|}+\frac{r_e+r_i}{4\rho|p|}(r_e-r_i)\,.
$$
Since
$$
|c^-|+|c^+|\geq |c^-+c^+|=2|\rho|,
$$
and
$$
r_e=r_i+(r_e-r_i)\leq |p|+(r_e-r_i),
$$
we have that 
$$
\frac{p}{|p|}\cdot \nu_p\leq -1 +\frac{r_e-r_i}{2\rho}+\frac{(r_e-r_i)^2}{4\rho^2}\leq -1 +\frac{r_e-r_i}{\rho}\,,
$$
as required.
\end{proof}

Now we are ready to prove Corollary \ref{main2}.

\smallskip
\begin{proof}
\emph{Step 1: $S$ is diffeomorphic to a sphere}. 
In view of Theorem \ref{main}, there exists $\tilde \ep$ and $C$ such that if 
${\rm osc}(H)<\tilde\ep$, then \eqref{Bri Om Bre} and \eqref{stability radii} hold. We may assume the concentric balls $B_{r_e}$ and $B_{r_i}$ centred in the origin.
Let 
\begin{equation}\label{SUKA}
\ep =\min\left\{\tilde \ep,\frac{\rho}{2C}\right\}\,.
\end{equation}
Hence the assumptions in Lemma \ref{pre} are satisfied. We consider the map
$
\varphi\colon S\to \partial B_{r_i},
$
defined by
$$
\varphi(p)=r_i\frac{p}{|p|}\,.
$$
We show that $\varphi$ a diffeomorphism.  It is clear that $\varphi$ is
smooth. Since $B_{r_i}$ is contained in the bounded domain enclosed by $S$, then $\varphi$ is surjective. Indeed, if $\zeta \in \partial B_{r_i}$, then
$$
\begin{cases}
\dist_S(\zeta) \leq 0\,, \\
\dist_S((r_e-r_i)\zeta) \geq 0, 
\end{cases}
$$
and, by continuity, there exists a $t\geq 0$ such that $\dist_S((1+t)\zeta)=0$, i.e. $\zeta \in \varphi(S)$. Hence, assumption \eqref{H quasi const} plays a role only for proving the injectivity of $\varphi$. Let $p,q\in S$ and assume by contradiction that $\varphi(p)=\varphi(q)$. Then we may assume that $|p|<|q|$. Let $c^+=p-\rho \nu_p$ be the center of the exterior touching ball to $S$ at $p$. Since $p/|p|=q/|q|$, we have
$$
|q-c^+|^2=\left|(|q|-|p|)\frac{p}{|p|}+\rho \nu(p)\right|^2=(|q|-|p|)^2+\rho^2+2\rho (|q|-|p|)\frac{p}{|p|}\cdot \nu_p\,.
$$
From Lemma \ref{pre} and since $|q|-|p|\leq r_e-r_i$, we have that
$$
|q-c^+|^2\leq (r_e-r_i)^2+\rho^2+2\rho(r_e-r_i)\left(-1+\frac{r_e-r_i}{\rho}\right)=\rho^2-(r_e-r_i)\left(2 \rho - 3 (r_e-r_i)\right).
$$
The choice of $\ep$, as in \eqref{SUKA} implies that $|q-c^+|<\rho$ which gives a contradiction.

\smallskip 
\emph{Step 2: proof of \eqref{Lipschitz_bound}}. We denote by $F\colon \partial B_{r_i}\to S$ the inverse of the map $\varphi\colon S\to \partial B_{r_i}$ considered in the first step.  We can write $F(\zeta)=\zeta+\Psi(\zeta)\frac{\zeta}{r_i}$ for every $\zeta$ in $\partial B_{r_i}$ and from Step 1 and Theorem \ref{main} it follows that $\|\Psi\|_{C^0(\partial B_{r_i})} \leq C \oscH$. In order to prove a quantitative bound on the $C^0$-norm of the derivatives of $\Psi$, we work in the same fashion as in the proof of Lemma \ref{lemma change normal}. 

Let $\zeta$ be a fixed point on $\partial B_{r_i}$ and set $p=F(\zeta)$ (i.e. $\zeta=r_ip/|p|$). Let $T_{\zeta}$ and $T_{p}$ be the tangent spaces to $\partial B_{r_i}$ at $\zeta$ and to $S$ at $p$, respectively.  
 We can locally write $S$ around $p$ as 
$$
q=p+x+u(x)\nu_p \,,
$$ 
where $x$ belongs to a small neighborhood of the origin $O$ and $u$ is a $C^2$ map satisfying $u(O)=0$ and $\nabla u(O)=0$.  Without loss of generality, we can assume that $\zeta=r_{i}e_{n+1}$ so that 
$$
T_{\zeta}=\{x\in \RR^{n+1}\,\,:\,\, x_{n+1}=0\} \,,
$$
and we locally write $\partial B_{r_i}$ as $\zi'=\zi + x + \eta(x) \nu_\zi$, where $\eta(x)=r_i-\sqrt{r_i^2-|x|^2}$. 

As in the proof of Lemma \ref{lemma change normal}, we can chose an orthogonal matrix $A\in {\rm SO}(n+1)$ satisfying $A(\zeta)=-r_i\nu_p$ (we recall that $\nu_\zi= -\zi/r_i$), and we can locally write 
\begin{equation}\label{finaleligure}
p+Ax+u(Ax)\nu_p=p+x+v(x)\nu_{\zeta}\,;
\end{equation}
furthermore, $A$ is such that 
\begin{equation} \label{A-Id}
|A-I| \leq 2 \sqrt{1-\nu_\zi \cdot \nu_p} \,.
\end{equation}
We firstly prove that 
\begin{equation} \label{abba}
\partial_{x_k}\psi(O)=-\frac{1}{r_i}\partial_{x_k} v(O) \,,\quad k=1,\dots,n \,.
\end{equation}
Indeed, by setting $\psi=\Psi\circ \eta$, we have  
$$
p+x+v(x) \nu_\zi =\eta(x) - \psi(x) \nu_{\eta(x)} \,,
$$ 
which implies 
$$
p\cdot  \nu_{\eta(x)}+x\cdot \nu_{\eta(x)}+v(x) \nu_\zi \cdot \nu_{\eta(x)}-\eta(x)\cdot \nu_{\eta(x)}=-\psi(x)\,,
$$
i.e.
$$ 
\frac{1}{r_i}p\cdot \eta(x)+\frac{1}{r_i}x\cdot \eta(x)+\frac{1}{r_i}v(x) \nu_\zi \cdot \eta(x)-r_i=\psi(x)\,,
$$
where we have used that $\nu_{\eta(x)}=-\eta(x)/r_i$.
From $\eta(O)=\zeta$ and $v(O)=0$ we obtain \eqref{abba}. 

Now, we give a bound on the derivatives of $v$ at $O$ and in terms of the difference $r_e-r_i$. We notice that \eqref{finaleligure} implies 
$$
v(x)=(A-I)x\cdot \nu_{\zeta}+u(Ax)\nu_p\cdot\nu_\zeta \,,
$$
and, since $|\nabla u(O)|=0$, we obtain that 
$$
|\partial_{x_k} v(O)| \leq |A-I |  \,,\quad k=1,\dots,n \,.
$$
From \eqref{A-Id} and Lemma \ref{pre} we obtain that
$$
|\partial_{x_k} v(O)| \leq 2 \sqrt{\frac{r_e-r_i}{\rho}} \,,\quad k=1,\dots,n \,,
$$
and from \eqref{stability radii} and \eqref{abba} we find \eqref{Lipschitz_bound} and we conclude.

\end{proof}

\begin{remark} \label{remark ros} \normalfont
As emphasized in the Introduction, if we assume that $\rho$ is not bounded from below, it is possible to construct a family of closed surfaces embedded in $\RR^3$, not diffeomorphic to a sphere, with $\osc(H)$ arbitrarly small and such that \eqref{stability radii} fails. For instance one can consider the following example, suggested us by A. Ros, done by gluing almost pieces of unduloids. 

\begin{figure}[h]
\begin{tikzpicture}
\draw [thick,domain=8:262] plot ({cos(\x)}, {sin(\x)});
\draw [thick,domain=278:352] plot ({cos(\x)}, {sin(\x)});

\draw [thick,domain=98:352] plot ({cos(\x)}, {-2.1+sin(\x)});
\draw [thick,domain=8:82] plot ({cos(\x)}, {-2.1+sin(\x)});

\draw [thick,domain=-82:172] plot ({2.1+cos(\x)}, {sin(\x)});
\draw [thick,domain=188:262] plot ({2.1+cos(\x)}, {sin(\x)});

\draw [thick,domain=98:172] plot ({2.1+cos(\x)}, {-2.1+sin(\x)});
\draw [thick,domain=188:360+82] plot ({2.1+cos(\x)}, {-2.1+sin(\x)});

\draw [thick] ({cos(8)}, {sin(8)}) ..  controls(2.1/2,0).. ({2.1+cos(172)},{sin(172)});
\draw [thick] ({cos(-8)}, {sin(-8)}) ..  controls(2.1/2,0).. ({2.1+cos(188)},{sin(188)});

\draw [thick] ({cos(262)}, {sin(262)}) ..  controls(0,-2.1/2).. ({cos(98)},{-2.1+sin(98)});
\draw [thick] ({cos(-8)}, {sin(-8)}) ..  controls(2.1/2,0).. ({2.1+cos(188)},{sin(188)});
\draw [thick] ({cos(278)}, {sin(278)}) ..  controls(0,-2.1/2).. ({cos(82)},{-2.1+sin(82)});
\draw [thick] ({cos(8)}, {-2.1+sin(8)}) ..  controls(2.1/2,-2.1).. ({2.1+cos(172)},{-2.1+sin(172)});
\draw [thick] ({cos(-8)}, {-2.1+sin(-8)}) ..  controls(2.1/2,-2.1).. ({2.1+cos(188)},{-2.1+sin(188)});
\draw [thick] ({2.1+cos(82)}, {-2.1+sin(82)}) ..  controls(2.1,-2.1/2).. ({2.1+cos(-82)},{sin(-82)});
\draw [thick] ({2.1+cos(98)}, {-2.1+sin(98)}) ..  controls(2.1,-2.1/2).. ({2.1+cos(-98)},{sin(-98)});
\end{tikzpicture}
\end{figure}

\end{remark}


\begin{thebibliography}{99}


\bibitem{ABR} {\sc A.~Aftalion, J.~Busca, W.~Reichel},
Approximate radial symmetry for overdetermined boundary value problems,
{\em Adv. Diff. Eq.}  {\bf 4} no. 6 (1999), 907--932.

\bibitem{Al0} {\sc A.~D.~Alexandrov}, Uniqueness theorems for surfaces
in the large II, {\em Vestnik Leningrad Univ.} {\bf 12}, no. 7 (1957), 15--44.
(English translation: {\em Amer. Math. Soc. Translations, Ser. 2}, {\bf 21} (1962),
354--388.) 

\bibitem{Al1} {\sc A.~D.~Alexandrov}, Uniqueness theorems for surfaces
in the large V, {\em Vestnik Leningrad Univ.} {\bf 13}, no. 19 (1958), 5--8.
(English translation: {\em Amer. Math. Soc. Translations, Ser. 2}, {\bf 21} (1962),
412--415.)

\bibitem{Al2}  {\sc A.~D.~Alexandrov}, A characteristic property of spheres, {\em Ann. Mat. Pura Appl.}, {\bf 58} (1962), 303--315.

\bibitem{All} {\sc W.~K.~Allard}, On the first variation of a varifold, {\em Ann. Math.}, {\bf 95} (1972), 417--491.

\bibitem{Alm} {\sc F.~J.~Jr.~Almgren}, {\em Plateaus problem. An invitation to varifold geometry}, Mathematics
Monograph Series, W. A. Benjamin, Inc., New York-Amsterdam, 1966.

\bibitem{Ar} {\sc R.~Arnold}, On the Alexandrov-Fenchel Inequality and the Stability of the Sphere, {\em Monatsh. Math.}, {\bf 155 (1993)}, 1--11.

\bibitem{BCN} {\sc H.~Berestycki, L.~A.~Caffarelli, L.~Nirenberg}, Inequalities for second-order elliptic equations with applications to unbounded domains I, {\em Duke Math. J.}, {\bf 81} (1996), no. 2, 467--494.

\bibitem{BNST} {\sc B.~Brandolini, C.~Nitsch, P.~Salani, C.~Trombetti}, On the stability of the Serrin problem, {\em J. Diff. Equations} {\bf 245} (2008), 1566--1583.


\bibitem{Br} {\sc S.~Brendle}, Constant mean curvature surfaces in warped product manifolds, {\em Publ. Math. Inst. Hautes \'Etudes Sci.} {\bf 117} (2013), 247--269.

\bibitem{BE} {\sc S.~Brendle, M.~Eichmair}, Isoperimetric and Weingarten surfaces in the Schwarzschild manifold, {\em J. Diff. Geom.}, {\bf 94} (2013), 387--407.

\bibitem{CFSW} {\sc X.~Cabr\'e, M.~Fall, J.~Sola-Morales, T.~Weth}, Curves and surfaces with constant nonlocal
mean curvature: meeting Alexandrov and Delaunay. Preprint, 2015. {\tt Arxiv:1503.00469.}

\bibitem{CGS} {\sc L.~Caffarelli, B.~Gidas, J.~Spruck}, Asymptotic symmetry and local behavior of semilinear elliptic equations
with critical Sobolev growth. {\em Comm. Pure Appl. Math.}, {\bf 42} (1989), 271--297.

\bibitem{CS}  {\sc L.~Caffarelli, S.~Salsa}, {\em A geometric approach to free boundary problems.} Graduate Studies in Mathematics, 68. American Mathematical Society, Providence, RI, 2005.

\bibitem{CY} {\sc D.~Christodoulou, S.~T.~Yau}, Some remarks on the quasi-local mass, {\em Contemp. Math.} {\bf 71} (1988), 9--14.

\bibitem{CFMN} {\sc G.~Ciraolo, A.~Figalli, F.~Maggi, M.~Novaga}, Rigidity and sharp stability estimates for hypersurfaces with constant and almost-constant nonlocal mean curvature. To appear in {\em J. Reine Angew. Math.} (Crelle's Journal). {\tt Arxiv:1503.00653.}

\bibitem{CirMag} {\sc G.~Ciraolo, F.~Maggi}, On the shape of compact hypersurfaces with almost constant mean curvature. Preprint. {\tt ArXiv:1503.06674.}

\bibitem{CMS} {\sc G.~Ciraolo, R.~Magnanini, S.~Sakaguchi}, Symmetry of solutions of elliptic and parabolic equations with a level surface parallel to the boundary, {\em J. Eur. Math. Soc. (JEMS)} {\bf 17} (2015), 2789--2804.

\bibitem{CMS2} {\sc G.~Ciraolo, R.~Magnanini, S.~Sakaguchi}, Solutions of elliptic equations with a level surface parallel to the boundary: stability of the radial configuration, to appear in {\em J. Analyse Math}. {\tt ArXiv:1307.1257.}

\bibitem{CMV}  {\sc G.~Ciraolo, R.~Magnanini, V.~Vespri}, H\"{o}lder stability for Serrin's overdetermined problem. Preprint. {\tt ArXiv:1410.7791.}

\bibitem{CMV2}  {\sc G.~Ciraolo, R.~Magnanini, V.~Vespri}, Symmetry and linear stability in Serrin's overdetermined problem via the stability of the parallel surface problem. Preprint. 

\bibitem{DeLMul1} {\sc C.~De Lellis, S.~M\"{u}ller}, Optimal rigidity estimates for nearly umbilical surfaces.
{\em J. Differential Geom.} {\bf 69} (2005), no. 1, 75--110.  

\bibitem{DeLMul2}  {\sc C.~De Lellis, S.~M\"{u}ller}, A $C^0$ estimate for nearly umbilical surfaces.
{\em Calc. Var. Partial Differential Equations} {\bf 26} (2006), no. 3, 283--296.

\bibitem{DCL} {\sc M.~P.~Do Carmo, H.~B.~Lawson}, On the Alexandrov-Bernstein Theorems in
Hyperbolic Space, {\em Duke Math Journal} {\bf 50} (1983), 995--1003.

\bibitem{GNN} {\sc B.~Gidas, W.~M.~Ni, L.~Nirenberg}, Symmetry and related properties via the maximum principle, {\em Comm. Math. Phys.}, {\bf 68} (1979), 209--243.

\bibitem{GT} {\sc D.~Gilbarg, N.~S.~Trudinger}, {\em Elliptic partial differential equations of second order}, Springer-Verlag, Berlin-New York, 1977.

\bibitem{Gr} {\sc M.~Gromov}, Stability and pinching. {\em Geometry Seminars. Sessions on Topology and Geometry of Manifolds},  (Bologna, 1990), Univ. Stud. Bologna, Bologna, 1992, 55--97. 

\bibitem{HTY} {\sc W.-Y.~Hsiang, Z.-H.~Teng, W.-C.~Yu}, New examples of constant mean curvature immersions of $(2k-
1)$-spheres into Euclidean $2k$-space, {\em Ann. of Math. (2)} {\bf 117} (1983), 609--625.

\bibitem{HY} {\sc G.~Huisken, S.-T.~Yau}, Definition of center of mass for isolated physical systems and unique foliations by stable spheres with constant mean curvature, {\em Invent. Math.} {\bf 124} (1996), no. 1-3, 281--311.

\bibitem{Ko} {\sc N.~J.~Korevaar}, Sphere theorems via Alexsandrov for constant Weingarten curvature hypersurfaces - Appendix to a note of A.~Ros, {\em J. Diff. Geom.}, {\bf 27} (1988), 221--223.

\bibitem{Kol} {\sc P.~Kohlmann}, Curvature measures and stability, {\em J. Geom.}, {\bf 68} (2000), 142--154.

\bibitem{Kou} {\sc D.~Koutroufiotis}, Ovaloids which are almost Spheres, {\em Comm. Pure Appl. Math.}, {\bf 24} (1971), 289--300.
    
\bibitem{KKS} {\sc N.~J.~Korevaar, R.~Kusner, B.~Solomon}, The structure of complete embedded surfaces with constant
mean curvature, {\em J. Diff. Geom.}, {\bf 30} (1989), 465--503.
    
\bibitem{KMPS} {\sc N.~J.~Korevaar, R.~Mazzeo, F.~Pacard, R.~Schoen.} Refined asymptotics for constant scalar curvature metrics with isolated singularities, {\em Invent.Math.}, {\bf 135} (1999), 233--272 .

\bibitem{La} {\sc U.~Lang}, Diameter Bounds for Convex Surfaces with Pinched Mean Curvature, {\em Manuscripta Math.}, {\bf 86} (1995), 15--22.


\bibitem{Li1} {\sc C.~Li}, Mononicity and symmetry of solutions of fully nonlinear elliptic equations on bounded domains, 
{\em Commun. Part. Diff. Eq.}, {\bf 16} (1991), 491--526.

\bibitem{Li2} {\sc C.~Li}, Mononicity and symmetry of solutions of fully nonlinear elliptic equations on unbounded domains,
{\em Commun. Part. Diff. Eq.}, {\bf 16} (1991), 585--615.

\bibitem{Me} {\sc W.~Meeks III}, The topology and geometry of embedded surfaces of constant mean curvature, {\em J. Diff. Geom.}
{\bf 27} (1988), 539--552.

\bibitem{Mo} {\sc S.~Montiel}, Unicity of constant mean curvature hypersurfaces in some Riemannian
manifolds, {\em Indiana Univ. Math. J.}, {\bf 48} (1999), 711--748.

\bibitem{Moo} {\sc J.~D.~Moore}, Almost Spherical Convex Hypersurfaces, {\em Trans. Amer. Math. Soc.}, {\bf 180} (1973), 347--358.

\bibitem{MR} {\sc S.~Montiel, A.~Ros}, {\em Compact hypersurfaces: The Alexandrov theorem for higher order mean curvatures},
{\rm Pitman Monographs and Surveys in Pure and Applied Mathematics} {\bf 52} (1991), 279--296.

\bibitem{Re} {\sc R.~Reilly}, Applications of the Hessian operator in a Riemannian manifold, {\em Indiana Univ. Math. J.}, {\bf 26} (1977), 459--472.


\bibitem{Ro1} {\sc A.~Ros}, Compact hypersurfaces with constant scalar curvature and a congruence theorem, {\em J. Diff. Geom.}, {\bf 27} (1988), 215--220.

\bibitem{Ro2} {\sc A.~Ros}, Compact hypersurfaces with constant higher order mean curvatures, {\em Rev. Math. Iber.}, {\bf 3} (1987), 447--453.

\bibitem{Se} {\sc J.~Serrin}, A symmetry problem in potential theory, {\em Arch. Rational Mech. Anal.} {\bf 43} (1971), 304-318.

\bibitem{Schn} {\sc R.~Schneider}, A Stability Estimate for the Alexandrov-Fenchel Inequality, with an Application to Mean Curvature, {\em Manuscripta Math.}, {\bf 69} (1990), 291--300.

\bibitem{Sc} {\sc R.~Schoen}, Uniqueness, symmetry, and embeddedness of minimal surfaces, {\em J. Diff. Geom.}, {\bf 18} (1983), 791--809.

\bibitem{We} {\sc H.~C.~Wente}, Counterexample to a conjecture of H. Hopf, {\em Pacific J. Math.} {\bf 121} (1986), 193--243.

\bibitem{Ya} {\sc S.-T.~Yau}, Submanifolds with constant mean curvature I, {\em Amer. J. Math.} {\bf 96} (1974), 346--366





\end{thebibliography}
\end{document}